\documentclass[11pt]{paper}

\usepackage{amssymb,amsfonts,amsmath,amsthm,mathrsfs}
\usepackage{mathtools}

\makeatletter
\newcommand{\subjclass}[2][1991]{%
  \let\@oldtitle\@title%
  \gdef\@title{\@oldtitle\footnotetext{#1 \emph{Mathematics subject classification.} #2}}%
}
\makeatother

\usepackage{graphicx, verbatim}
\usepackage{xcolor}
\usepackage{tikz}
\usetikzlibrary{shapes.geometric}
\usetikzlibrary{arrows.meta}
\usetikzlibrary{patterns}

\usepackage[a4paper, hmargin=3cm, vmargin={4cm, 4cm}]{geometry}
\usepackage[english]{babel}
\usepackage[colorlinks=true, linkcolor=black, citecolor=blue, urlcolor=blue]{hyperref}
\usepackage[final,tracking=true,kerning=true,spacing=true,
factor=1000,stretch=20,shrink=20,babel=true]{microtype}
\microtypecontext{spacing=nonfrench}
\usepackage{fancyhdr}

\newtheorem{theorem}{Theorem}
\newtheorem{lemma}{Lemma}

\newtheorem{corollary}{Corollary}

\newtheorem{claim}{Claim}

\newtheorem*{theorem*}{Theorem}
\newtheorem*{question*}{Question}
\newtheorem*{example*}{Example}
\newtheorem*{definition*}{Definition}
\newtheorem*{remark*}{Remark}

\def\done{{1\hskip-2.5pt{\rm l}}}

\newcommand{\bR}{\mathbb R}
\newcommand{\bC}{\mathbb C}
\newcommand{\bZ}{\mathbb Z}

\newcommand{\bN}{\mathbb N}
\newcommand{\bQ}{\mathbb Q}

\newcommand{\bP}{\mathbb P}

\newcommand{\cE}{\mathcal E}
\newcommand{\cZ}{\mathcal Z}
\newcommand{\cD}{\mathcal D}
\newcommand{\cC}{\mathcal C}
\newcommand{\cK}{\mathcal K}

\newcommand{\sW}{\mathscr W}

\newcommand{\eps}{\varepsilon}
\newcommand{\la}{\lambda}
\renewcommand{\phi}{\varphi}

\begin{document}

\title{
A measurable equivariant Weierstrass theorem
}

\author{Konstantin Slutsky, Mikhail Sodin, Aron Wennman}

\maketitle

\bigskip

\hfill{\textit{To Nick Makarov, a friend, colleague, and remarkable mathematician}}

\abstract{ This paper is a prequel to our recent work, ``Equivariant Borel liftings in complex
  analysis and PDE'' (arXiv:2507.12058). While the results presented here were established in that
  work in a more general and abstract setting, the purpose of this paper is to provide a direct
  proof of the equivariant Weierstrass theorem. It states that there exists a Borel map assigning to
  each non-periodic positive divisor $\Lambda$ an entire function $F_\Lambda$ such that the divisor
  of zeroes of $F_\Lambda$ is $\Lambda$ and such that $F_{\Lambda-w}(z) = F_\Lambda (z+w)$,
  $w\in\bC$. In general, non-periodicity cannot be omitted, and Borel measurability cannot be
  strengthened to continuity. The two key ingredients are the Runge approximation theorem and the
  existence of ``Borel toasts'', which are Borel counterparts of Rokhlin towers from ergodic theory.

We do not assume prior knowledge of descriptive set theory and have aimed to make
the exposition self-contained, aside from several results taken from graduate textbooks.
}

\section{Introduction and statement of main results}

Our starting point is the following question. Let $\cE$ denote the space of entire functions
of one complex variable that do not vanish identically, let $\cD$ denote the space of divisors (discrete multisets) in $\bC$,
and let \[ Z\colon \cE\to \cD \] be the map which assigns to each $F\in \cE$ its zero set $Z_F\in\cD$.
The Weierstrass theorem states that the map $Z$ has a right inverse, i.e., a map
\[ W\colon \cD\to \cE\]
satisfying $Z \circ W = {\sf id}_{\cD}$. Note that the additive group $\bC$ acts on $\cE$ and $\cD$
by translations \[ T_w F(z)=F(z+w), \quad F\in \cE, \] and
\[ T_w \Lambda = \Lambda - w, \quad \Lambda\in \cD,\]
and that the zero map
$Z$ respects these actions, $T_w Z = Z  T_w$, $w\in\bC$. We will be concerned with the following question.

\begin{question*} Does there exist an equivariant Weierstrass map, i.e., a function $\sW\colon \cD\to \cE$ such that
\[ Z \circ \sW = {\sf id}_{\cD} \ {\rm and}\  T_w \sW = \sW  T_w\]
for all $w\in\bC$?
\end{question*}

An evident obstacle is that there exist doubly periodic divisors in $\cD$, while there are no
non-constant doubly periodic entire functions. So, for the time being, we restrict ourselves to the
free part $\cD_0$ of $\cD$, which consists of non-periodic divisors.  Then the equivariant map
$\sW\colon \cD_0\to \cE$ can be readily constructed using the axiom of choice, which allows us to
select a representative $\Lambda_\mathcal O$ in each $\bC$-orbit $\mathcal O\subset\cD_0$. Since the
action $\bC\curvearrowright \cD_0$ is free, for any $\Lambda\in\mathcal O$, there is a unique
$w\in\bC$ such that $\Lambda=T_w \Lambda_\mathcal O$.  Then, we apply the classical Weierstrass map
$W$ to $\Lambda_\mathcal O$, and set \[ \sW_\Lambda = T_{w} W_{\Lambda_\mathcal O}. \]

The spaces $\cE$ and $\cD$ are endowed with natural topologies, namely, the topology of uniform
convergence on compact sets and the topology of vague convergence of Radon measures, which turn
these spaces into Polish spaces.  The actions $\bC\curvearrowright \cE$ and
$\bC\curvearrowright \cD$ are continuous and the map $Z$ is also continuous (by Hurwitz' theorem).
Answering a question by Eremenko, Remling~\cite{Remling} constructed a continuous right inverse to $Z$.
It is therefore natural to ask whether there exists a {\em continuous equivariant} Weierstrass map
from $\cD_0$ to $\cE$.  A brief reflection shows that such a map cannot exist. Indeed, $\cD_0$
contains non-trivial invariant compact sets\footnote {Perhaps, the simplest way to construct such
  compact invariant sets is to take a non-periodic almost periodic subset of $\bC$ with respect to
  the uniform transportation distance and then to take the closure of its $\bC$-orbit.  Recall that
  given two discrete sets $a=\{a_k\}_k$ and $b=\{b_k\}_k$ in $\bC$, the uniform transportation
  distance between them is given by
\[
\operatorname{dist}(a, b) =  \inf_\sigma  \sup_k |a_k - b_{\sigma(k)}|,
\]
where the infimum is taken over all bijections $\sigma\colon \bN\to \bN$.  See
Favorov--Kolbasina~\cite{FavorovKolbasina} for the details.  }; let $K$ be one of them. Then, by
continuity of $\sW$, the image $\sW(K)$ is a compact subset of $\cE$, and by equivariance of $\sW$,
it is also invariant. This contradicts the Liouville theorem since $\cE$ cannot contain non-trivial
(that is, containing non-constant entire functions) translation-invariant compact subsets.

We therefore abandon continuity and pass to Borel measurability. The spaces $\cE$ and $\cD$ are
equipped with the Borel $\sigma$-algebras generated by their open sets. Our main result asserts the
existence of a measurable equivariant Weierstrass map:
\begin{theorem}\label{thm:main}
There exists a Borel measurable equivariant map
$\sW\colon \cD_0 \to \cE$ such that \[ Z \circ \sW = {\sf id}_{\cD_0}.\]
\end{theorem}
An immediate corollary to Theorem~\ref{thm:main} is that the map $\sW$ lifts any
translation-invariant non-periodic point process on $\bC$ to a random entire function with a
translation-invariant distribution\footnote{ This corollary was proved earlier in unpublished work
  by M.~Sodin, A.~Wennman, and O.~Yakir.  The proof follows closely the approach of
  Weiss~\cite{Weiss} and yields, for each invariant probability measure $\bP$ on $\cD_0$, an
  equivariant Weierstrass map $\mathscr{W}^{\bP}$, defined on a set of full $\bP$-measure. One
  motivation for~\cite{SSW} was to understand whether there exists a single equivariant lift that
  works simultaneously for all non-periodic point processes.  }.  For instance, this yields the
existence of a random entire function with a translation-invariant distribution whose zero set is a
two-dimensional Poisson point process. It follows from~\cite{BGLS} that such entire functions always
have quite wild behaviour.

Curiously, the one-periodic case (i.e., the case of one independent period) is different.  Denote by
$\cE_{\sf per}$ and $\cD_{\sf per}$ the one-periodic parts of $\cE$ and $\cD$, respectively.  The
restriction of the divisor map $Z$ to $\cE_{\sf per}$ is a continuous surjection\footnote{The
  exponential map $z\mapsto e^{2\pi {\rm i}z/T}$, where $T$ is an independent period, transfers the
  question to surjectivity of the divisor map for analytic functions in the punctured plane
  $\bC\setminus\{0\}$, which is a classical result for arbitrary domains in $\bC$.  }
\[ Z\colon \cE_{\sf per}\to\cD\setminus \cD_0. \]
\begin{theorem}\label{thm:1-periodic}
There is no Borel measurable equivariant map $\sW\colon \cD_{\sf per}\to \cE_{\sf per}$ such that
$Z\circ \sW = {\sf id}_{\cD_{\sf per}}$.
\end{theorem}

Another non-existence result concerns the derivative $Df(z)=f'(z)$. To avoid a minor caveat, we
assume here, as well as in Section~\ref{sect:primitive}, that $\cE$ is the space of all entire
functions, including the one which identically vanishes.  Then the operator $D\colon \cE\to \cE$ is
a continuous operator, which respects the action $\bC \curvearrowright\cE$, that is, $T_w D = D T_w$
for $w\in \bC$. Nevertheless,
\begin{theorem}\label{thm:primitive}
There exists no Borel measurable equivariant right inverse to the differentiation $D$.
\end{theorem}

\subsubsection*{Related work}
The study of measurable entire functions began with a theorem of Benjamin Weiss~\cite{Weiss}:
\begin{theorem*}[Weiss]
Let $(X, \mathscr L, \bP)$ be a standard probability space, and $\tau\colon \bC\curvearrowright X$
a measure-preserving free action. Then, there exists a measurable map
$f\colon X\to \bC$  such that, for $\bP$-a.e. $x\in X$, the function $f_x (z) =  f(\tau_z x)$
is non-constant entire.
\end{theorem*}
\noindent Weiss calls the function $f$ {\em a measurable entire function}. Pushing forward the measure $\bP$
under the map $x \mapsto f_x$ yields a non-trivial translation-invariant probability measure on the
space $\cE$ of entire functions.
The proof in~\cite{Weiss} adapts to the Borel category\footnote{
The relation between the Lebesgue and Borel categories was clarified by Ramsay~\cite{Ramsey}, who
showed that every measurable action of a locally compact group on a standard Lebesgue space can be realized
as a Borel action on an invariant subset of full measure. Consequently, “from the measure-theoretic
viewpoint, no generality is lost by studying only Borel actions”~\cite[p.~340]{Ramsey}.
}.
The only essential modification is that  instead of using Rokhlin towers to cover orbits by compact sets with
connected complements, one employs the notion of {\em Borel toasts} (see Section~\ref{subsect:toasts}):
\begin{theorem*}[Borel version of the Weiss theorem]
Let $(X, \mathscr B)$ be a standard Borel space and $\tau\colon \bC\curvearrowright X$ be
a free action. Then there exists a measurable map  $f\colon X\to \bC$  such that, for every $x\in X$,
the function $f_x (z) =  f(\tau_z x)$ is non-constant entire.
\end{theorem*}

Buhovsky, Gl\"ucksam, Logunov and Sodin~\cite{BGLS} and  Gl\"ucksam~\cite{Glucksam}
studied the growth properties of measurable entire functions. Recently, Gl\"ucksam and Weiss~\cite{GW} extended Weiss’ construction to measurable entire functions of several complex variables.

\smallskip
In~\cite{SSW}, a “little theory” is developed, providing further equivariant counterparts of classical
results from complex analysis and partial differential equations. That work also contains several examples
of operators on entire functions that, similarly to the differentiation operator $D$, admit no measurable equivariant right inverse.

\subsubsection*{Use of AI tools}
AI tools were used during the preparation of this paper to aid
with proofreading, editing, creating figures, and content critique. All
mathematical ideas are due to the authors.

\subsubsection*{Acknowledgement}
We are grateful to Benjamin Weiss and Oren Yakir
for several very interesting discussions during the work on this paper.

Konstantin Slutsky was partially supported by NSF grant DMS-2153981.
The research of Mikhail Sodin was partially supported by ISF Grants 1288/21, 2319/25
and by BSF Grants 202019, 2024142. Aron Wennman was partially supported by Odysseus
Grant G0DDD23N from Research Foundation Flanders (FWO).

\section{Preliminaries}
Given a topological space $X$, its \emph{Borel \(\sigma\)-algebra} is the \(\sigma\)-algebra of subsets
of $X$ generated by the open sets. Main references for the theory of Borel sets are the textbooks by Kechris~\cite{Kechris} and Srivastava~\cite{Srivastava}.

A \emph{Polish space} is a separable completely metrizable topological space.
A \emph{standard Borel space} is a pair \((X, \mathscr B)\), where \(\mathscr B = \mathscr B_X\)
is a \(\sigma\)-algebra on
\(X\) that coincides with the Borel \(\sigma\)-algebra associated with  some Polish topology on \(X\).

It is worth keeping in mind that if $(X, \mathscr B)$ is standard Borel and $Y\subset X$ is in
$\mathscr B$, then $(Y, \mathscr B |_{Y})$ is also standard Borel~\cite[13.4]{Kechris}.

\subsubsection*{What everyone wanted to know about standard Borel spaces}
By the Kuratowski theorem, {\em all uncountable standard Borel spaces are Borel isomorphic},
see~\cite[Section~3.3]{Srivastava}, \cite[Section~15B]{Kechris}.
In this work, we always tacitly assume that the standard Borel space is uncountable.
Another useful fact, which is also worth keeping in mind:
\begin{theorem*}[Kuratowski]
Let $X$ be standard Borel,  $Y$ be Polish, and $f\colon X\to Y$ be a Borel map.
Then there exists a Polish topology $\tau_f$ on $X$ generating the same Borel $\sigma$-algebra,
so that $f\colon (X, \tau_f) \to Y$ is continuous.

If $(X, \tau)$ was a Polish space, then the topology $\tau_f$ can be chosen to be stronger
than the original one, $\tau_f \supseteq \tau$.
\end{theorem*}
Two other fundamental results on Borel spaces, which we will be using, are the following ones:
\begin{theorem*}[Luzin--Suslin]
Let $X$, $Y$ be standard Borel spaces, and $f\colon X\to Y$ be Borel. If $A\subseteq X$ is Borel
and $f|_A$ is injective, then $f(A)$ is Borel and $f$ is a Borel isomorphism $A\to f(A)$.
\end{theorem*}
\begin{theorem*}[Luzin--Novikov Uniformization]
Let $X$, $Y$ be standard Borel spaces, and $B\subseteq X \times Y$ be Borel. If every section
$B_x = \{y\in Y\colon (x, y)\in B\}$ is countable, then there exists a Borel
uniformization $f\colon \operatorname{proj}_X(B)\to Y$
such that $\{(x, f(x))\colon x\in \operatorname{proj}_X(B)\}\subset B$ {\rm (}the graph of $f$ is in $B${\rm )}.
In particular, the projection $\operatorname{proj}_X(B) =\{x\in X\colon \exists y\in Y\ (x, y)\in B\}$
is Borel.

Moreover, $B=\bigsqcup_n B_n$, where each $B_n$ is a Borel graph
{\rm (}i.e., if $(x, y), (x, y')\in B_n$, then $y'=y${\rm )}.
\end{theorem*}
\noindent For the proofs see~\cite[Proposition~4.5.1 and Theorem~5.7.2]{Srivastava}
and~\cite[Theorems~15.1 and 18.10]{Kechris}.

\begin{remark*}
  \label{rem:uniformization-list}
{\rm   The sets \(B_{n}\) are graphs of Borel functions
  \(\widetilde{g}_{n}\colon  \operatorname{proj}_{X}(B_{n}) \to Y\).  We can define
  \[g_{n}(x) =
    \begin{cases}
      \tilde{g}_{n}(x) & \text{if \(x \in \operatorname{proj}_{X}(B_{n})\)}, \\
      f(x) & \text{otherwise}.
    \end{cases}
  \]
  Then \(g_{n}\colon \operatorname{proj}_{X}(B) \to Y\) are Borel functions that enumerate the
  sections, possibly with repetitions:
\[B = \{(x,g_{n}(x))\colon x \in \operatorname{proj}_{X}(B), n \in \bN\}.\]
}
 \end{remark*}
\noindent This remark helps us to work with Borel enumerations of the sections
$B_x = \{y \in Y\colon  (x,y) \in B\}$ (see Theorem~\ref{thm:L-N-ordered} in Appendix~\ref{AppB}).

\subsubsection*{Vietoris topology and Effros Borel space}
We will be using some basic facts about the Vietoris topology on compact sets and the corresponding
Borel structure \cite[pages 66--69, 97--98]{Srivastava}, \cite[pages 24--28, 75--77]{Kechris}.

By $\cK = \cK(\bR^d)$ we denote the space of non-empty compact sets in $\bR^d$ endowed with the Vietoris topology.
The basis of this topology is parameterized by open sets $U_0, U_1, \ldots , U_n \subseteq \bR^d$ and is given
by
\[
\bigl[ U_0; U_1,  \ldots , U_n \bigr]
= \bigl\{ K\in\cK\colon K\subset U_0, K \cap U_1 \ne \emptyset, \ldots ,
K\cap U_n \ne \emptyset \bigr\}.
\]
Equivalently, the Vietoris topology is induced by the Hausdorff metric on $\cK$
\begin{align*}
d_H(A, B) &= \inf\bigl\{\eps>0\colon A\subset B_{+\eps}, \, B\subset A_{+\eps} \bigr\} \\
&= \max\bigl\{ \max_{a\in A} \operatorname{dist}(a, B),  \, \max_{b\in B} \operatorname{dist}(b, A) \bigr\},
\quad A, B \in \cK,
\end{align*}
where $E_{+\eps}$ is the $\eps$-neighbourhood of the set $E\subset \bR^d$.
Endowed with this metric, $\cK$ is a Polish metric space.

The Borel structure on $\cK$ is called the Effros Borel structure. It is countably generated by the sets
$\bigl\{K\in \cK\colon K\cap U_n \ne \emptyset \bigr\}_n$, where $(U_n)_n$ is a countable base of open sets
in $\bR^d$.

By $\cK_0 \subset \cK$ we denote the subclass of compact sets with non-empty interior. It is an $F_\sigma$-subset of $\cK$, and hence, a Borel one.

By $\cK_* \subset \cK_0$ we denote the subclass of compact sets in $\bR^d$, $d\ge 2$,
with non-empty interior and connected complement. It is also Borel. Indeed, $\cK_* = \cK_0 \cap C$, where
\[
C = \bigl\{K\in\cK\colon \bR^d\setminus K \ {\rm is\ connected} \bigr\}.
\]
To see Borelness of $C$, we represent $C$ in the form
\[
C = \bigcap_{p, q \in \bQ^d} P_{p, q},
\]
where
\begin{align*}
P_{p, q} &=
\bigl\{
K\in\cK\colon p\in K\ {\rm or\ } q\in K\ {\rm or\ } \exists \ {\rm a\ path\ in\ } \bR^d\setminus K\
{\rm from\ } p \ {\rm to\ } q \bigr\} \\
&= \bigl\{K\in\cK\colon  p\in K \bigr\} \cup    \bigl\{K\in\cK\colon q\in K \bigr\}
\cup \ \bigcup_{\mathclap{\Gamma\in\operatorname{Path}(p, q)}} \ \bigl\{K\in\cK\colon K \cap \Gamma = \emptyset \bigr\}.
\end{align*}
The first and the second sets on the RHS are closed in $\cK$, while the third one is open. Hence, $C$ is Borel,
and therefore, $\cK_*$  is Borel as well.

\subsubsection*{Free Borel actions}
Fix a standard Borel space $X$ and a {\em free Borel action} $ T\colon \bR^d \curvearrowright X$.
Borelness means that the map
\[ \bR^d \times X \ni (w, x) \mapsto T_w x \in X \]
is Borel, and freeness means that $T_w x \ne x$, for each $x\in X \ {\rm and\ } w\ne 0$.  It is
worth mentioning that by a deep equivariant generalization of the Kuratowski theorem due to Becker
and Kechris any Borel action of a Polish group $G$ on a standard Borel
space $X$ can be made continuous by choosing an appropriate Polish topology on $X$ that generates its Borel structure; see~\cite[Chapter 5]{Becker-Kechris} and~\cite[Lecture~II]{Kechris-Lectures}.
So, without loss of generality, we may assume that the action is continuous.

Each orbit of the action can be viewed as an affine copy of $\bR^d$ with no distinguished origin, carrying
all translation-invariant structures of $\bR^d$, such as the Euclidean distance and the Lebesgue measure.
For any two $x, y\in X$ belonging to the same orbit,
denote by $\rho (x, y)$ the {\em cocycle} of the action $T$, namely, the unique $w\in \bR^d$ such that
$y=T_w x$. One can think about $\rho (x, y)$ as the position of the point $y$ viewed from $x$.
Then $|\rho (x, y)|$ (the length of the vector $\rho (x, y)$) coincides with the Euclidean distance
between $x$ and $y$ viewed as points of the affine copy of $\bR^d$.

By $E_T = \{ (x, y)\in X\times X\colon \exists w\in\bR^d\ T_w x = y\}$ we denote the orbit equivalence relation. The map
\[
\Phi\colon X\times \bR^d \ni (x, w) \mapsto (x, T_w x)\in X\times X
\]
is injective and Borel. Hence, by the Luzin--Suslin theorem, $E_T$ is Borel as the range of $\Phi$.
Furthermore,
\[
\Phi^{-1}\colon E_T\ni (x, y)\mapsto (x, \rho(x, y))\in X\times \bR^d
\]
is Borel. Hence, the cocycle is a Borel function $\rho\colon E_T\to \bR^d$.

In this work, $T$ mostly will be a free continuous action of $\bC$ by translations on the space $\cD_0$
of non-periodic divisors or on the space $\cE_0$ of non-periodic entire functions.
We will also need an action of the
multiplicative group of non-vanishing entire functions $\cE^\times$ on $\cE$ by multiplication:
\begin{equation}\label{eq:mult_action}
(T_H F) (z) = H(z)F(z), \quad z\in \bC, \ H\in\cE^\times, \ F\in \cE.
\end{equation}
This action is also continuous and free.
The entire functions $F$ and $G$ belong to the same orbit of this action if and only if they have the same zero divisors.  In this case, the corresponding cocycle is $\rho_{\cE^\times} (F, G) = G/F$.

\section{Borel cross-sections and toasts}
In this section, we introduce the main tools, Borel cross-sections and Borel toasts, on which the proof of Theorem~\ref{thm:main} rests.

\subsection{Cross-sections}
A {\em Borel cross-section} of the action $ T\colon \bR^d \curvearrowright X$ is a Borel set
$\cC\subset X$ that hits each orbit of the action in a non-empty discrete set, that is, for each
$x\in X$ and each compact set $K\subset\bR^d$, the set $T_K x \cap \cC$ is finite and non-empty if
\(K\) is large enough.  For instance, let $X=\cD$ be the space of discrete multisets in $\bR^d$ and
$\bR^d$ acts on $\cD$ by translations.  Consider the subspace $\cD_0$ of non-periodic divisors, that
is, the free part of the action.  As a $G_\delta$ subset of the Polish space $\cD$ endowed with its
inherited topology, it is also a standard Borel space~\cite[Proposition~3.3.7]{Srivastava}. Then
\[ \cC_0 = \{\Lambda\in \cD_0\colon 0\in{\rm spt}(\Lambda)\} \]
is a closed cross-section of $ \cD_0$ (here, ${\rm spt}(\Lambda)$ denotes the support
of the divisor $\Lambda$).

A cross-section $\cC$ is {\em uniformly separated}
(a.k.a. {\em lacunary}) if there exists
a ball $B\subset \bR^d$ such that, for any two distinct points $c\ne c'$ in $\cC$,
$ T_B c  \cap T_B c'= \emptyset$.
The cross-section $\cC$ is {\em relatively dense} (a.k.a.\, {\em cocompact}) if there exists a ball
$B'\subset\bR^d$ such that $T_{B'} \cC= X$.
Requiring that $\cC$ be both uniformly separated and relatively dense is
equivalent to demanding that $\cC$ be a Delone set on each orbit.
When the radii $r$ of $B$ and $R$ of $B'$ are relevant, we call the cross-section $r$-separated
and $R$-dense, respectively.

The following theorem is a combination of results of Kechris~\cite{Kechris1992, Kechris-Lectures},
Kechris--Solecki--To\-do\-rcevic~\cite{KST}, Boykin--Jackson~\cite{BJ}, and Slutsky~\cite{Slutsky2017}.

\begin{theorem}\label{thm:cross-sections}
Let  $\bR^d \curvearrowright X$ be a free Borel flow on a standard Borel space $X$, and let
$1\le r_1 < r_2 < \ldots$
be an increasing sequence of positive reals tending to infinity. Then there exists
a sequence of Borel cross-sections $(\cC_n)_{n\in\bN}$ such that
\begin{enumerate}
\item for every $n$, the cross-section $\cC_n$ is $r_n$-separated and $2r_n$-dense;
\item for every $x\in X$ and $\eps>0$, there are infinitely many $n$ for which
$|\rho (x, c_n)|< \eps r_n$ for some $c_n\in\cC_n$;
\item for every $m$ and $n$, and every $c_m\in \cC_m$ and $c_n\in\cC_n$ lying on the same orbit,
$\rho(c_m, c_n)\in\bQ^d$.
\end{enumerate}
\end{theorem}
In Appendix~\ref{app:rec-cross-sections}, we will provide the full proof of
Theorem~\ref{thm:cross-sections} except the very fact of existence of a Borel cross-section
of the action $T$ (Kechris~\cite{Kechris1992, Kechris-Lectures}). The reason for this omission is that
in the proof of Theorem~\ref{thm:main} we need only the action $T\colon\bC \curvearrowright  \cD_0$ by translations, for which, as we have already mentioned, the existence of a Borel cross-section is
evident.

\subsection{Toasts}\label{subsect:toasts}
\begin{figure}[ht]
\centering
\begin{tikzpicture}[scale=1.0,line join=round,
    bluetile/.style={thick,blue!60,preaction={fill=blue!5},
      pattern=north east lines,pattern color=blue!30},
    blacktile/.style={thin,black!80,fill=white},
    fibn/.style={blue!70,line width=0.6pt},
    fibm/.style={black!80,line width=0.6pt},
    leftregion/.pic={
      \draw[bluetile] plot[smooth cycle,tension=0.8] coordinates
        {(-1.1,0.05) (-0.75,0.55) (-0.05,0.62) (0.75,0.5)
         (1.15,0.0) (0.7,-0.55) (-0.15,-0.62) (-0.85,-0.5)};
      \draw[blacktile] plot[smooth cycle,tension=0.8] coordinates
        {(0.0,0.08) (0.2,0.35) (0.55,0.38) (0.78,0.12)
         (0.6,-0.25) (0.25,-0.32) (0.02,-0.18)};
      \fill[blue!70] (-0.7,0.05) circle (1.3pt);
      \fill (0.38,0.02) circle (1.1pt);},
    rightregion/.pic={
      \draw[bluetile] plot[smooth cycle,tension=0.82] coordinates
        {(-1.4,0.12) (-1.0,0.72) (-0.2,0.86) (0.7,0.8) (1.5,0.52)
         (1.72,-0.1) (1.25,-0.72) (0.3,-0.85) (-0.6,-0.78) (-1.28,-0.42)};
      \draw[blacktile] plot[smooth cycle,tension=0.8] coordinates
        {(-0.35,0.42) (-0.2,0.68) (0.12,0.7) (0.28,0.46) (0.12,0.22) (-0.2,0.22)};
      \draw[blacktile] plot[smooth cycle,tension=0.8] coordinates
        {(0.1,-0.12) (0.4,0.16) (0.9,0.12) (1.15,-0.22)
         (0.95,-0.6) (0.5,-0.66) (0.12,-0.45)};
      \fill[blue!70] (-1.0,0.15) circle (1.3pt);
      \fill (-0.03,0.48) circle (1.0pt);
      \fill (0.62,-0.26) circle (1.1pt);},
    leftregionB/.pic={
      \draw[bluetile] plot[smooth cycle,tension=0.8] coordinates
        {(1.26,-0.05) (1.02,0.52) (0.50,0.90) (-0.18,0.97) (-0.82,0.80) (-1.22,0.32)
         (-1.28,-0.28) (-0.92,-0.72) (-0.30,-0.90) (0.45,-0.82) (1.02,-0.50)};
      \draw[blacktile] plot[smooth cycle,tension=0.85] coordinates
        {(-0.12,0.02) (0.05,0.24) (0.42,0.30) (0.82,0.24) (1.02,0.00)
         (0.82,-0.24) (0.42,-0.30) (0.00,-0.22)};
      \fill[blue!70] (-0.7,0.05) circle (1.3pt);
      \fill (0.38,0.02) circle (1.1pt);},
    rightregionB/.pic={
      \draw[bluetile] plot[smooth cycle,tension=0.8] coordinates
        {(2.00,0.00) (1.70,0.58) (1.05,0.92) (0.30,0.96) (-0.45,0.85) (-1.12,0.58)
         (-1.50,0.05) (-1.35,-0.50) (-0.70,-0.86) (0.12,-0.95) (0.92,-0.90) (1.70,-0.56)};
      \draw[blacktile] plot[smooth cycle,tension=0.8] coordinates
        {(-0.34,0.30) (-0.22,0.62) (0.14,0.70) (0.36,0.50) (0.28,0.24) (-0.08,0.18)};
      \draw[blacktile] plot[smooth cycle,tension=0.8] coordinates
        {(0.06,-0.04) (0.30,0.18) (0.66,0.22) (1.04,0.10) (1.40,-0.16)
         (1.48,-0.50) (1.14,-0.74) (0.66,-0.80) (0.30,-0.60) (0.08,-0.34)};
      \fill[blue!70] (-1.0,0.15) circle (1.3pt);
      \fill (-0.03,0.48) circle (1.0pt);
      \fill (0.62,-0.26) circle (1.1pt);}]

  \draw[thin,black!85] (0,3.1) -- (7.5,3.1) -- (8.9,5.5) -- (1.4,5.5) -- cycle;
  \draw[thin,black!85] (0,-0.2) -- (7.5,-0.2) -- (8.9,2.2) -- (1.4,2.2) -- cycle;

  \pic at (2.6,4.3) {leftregion};
  \pic at (5.6,4.3) {rightregion};
  \pic at (2.6,1.0) {leftregionB};
  \pic at (5.6,1.0) {rightregionB};

  \draw[fibn] (2.02,6.40) .. controls (1.75,5.55) and (2.02,4.95) .. (1.90,4.35);
  \draw[fibn] (1.87,3.00) .. controls (1.98,2.45) and (1.80,1.70) .. (1.90,1.05);
  \draw[fibn] (1.90,-0.32) .. controls (1.84,-0.58) and (1.98,-0.80) .. (1.95,-0.98);
  \draw[fibm] (3.10,6.35) .. controls (2.82,5.55) and (3.08,4.95) .. (2.98,4.30);
  \draw[fibm] (2.95,3.00) .. controls (3.06,2.45) and (2.88,1.65) .. (2.98,1.00);
  \draw[fibm] (2.98,-0.32) .. controls (2.92,-0.58) and (3.05,-0.80) .. (3.02,-0.98);
  \draw[fibn] (4.72,6.40) .. controls (4.45,5.55) and (4.72,4.98) .. (4.60,4.45);
  \draw[fibn] (4.57,3.00) .. controls (4.68,2.45) and (4.50,1.75) .. (4.60,1.15);
  \draw[fibn] (4.60,-0.32) .. controls (4.54,-0.58) and (4.68,-0.80) .. (4.65,-0.98);
  \draw[fibm] (5.65,6.35) .. controls (5.42,5.60) and (5.66,5.20) .. (5.57,4.78);
  \draw[fibm] (5.54,3.00) .. controls (5.66,2.60) and (5.48,1.95) .. (5.57,1.48);
  \draw[fibm] (5.57,-0.32) .. controls (5.50,-0.58) and (5.47,-0.80) .. (5.45,-0.98);
  \draw[fibm] (6.30,6.30) .. controls (6.06,5.55) and (6.32,4.75) .. (6.22,4.04);
  \draw[fibm] (6.19,3.00) .. controls (6.30,2.40) and (6.12,1.40) .. (6.22,0.74);
  \draw[fibm] (6.22,-0.32) .. controls (6.30,-0.58) and (6.42,-0.80) .. (6.42,-0.98);

  \node[blue!70,font=\small] at (4.1,3.4) {$R_n(c)$};
  \node[blue!70,font=\footnotesize] at (4.8,4.46) {$c$};
  \node[black!80,font=\footnotesize,fill=white,fill opacity=0.75,text opacity=1,rounded corners=10pt] at (5.0,4.0) {$R_m(c')$};
  \node[black!85,font=\footnotesize] at (6.5,4.0) {$c'$};

  \node[blue!70,font=\small] at (1.95,-1.28) {$\mathcal C_n$};
  \node[black!80,font=\small] at (3.02,-1.28) {$\mathcal C_m$};
  \node[blue!70,font=\small] at (4.65,-1.28) {$\mathcal C_n$};
  \node[black!80,font=\small] at (5.45,-1.28) {$\mathcal C_m$};
  \node[black!80,font=\small] at (6.42,-1.28) {$\mathcal C_m$};
\end{tikzpicture}
\caption{Two sheets of a toast: blue tiles $R_n(c)$ with
centres $c\in \mathcal{C}_n$ and black subtiles $R_m(c')$ with centres $c'\in\mathcal{C}_m$,
\(m<n\), with the sections
$\mathcal C_n$ and $\mathcal C_m$ marked.}
\label{fig:toast}
\end{figure}
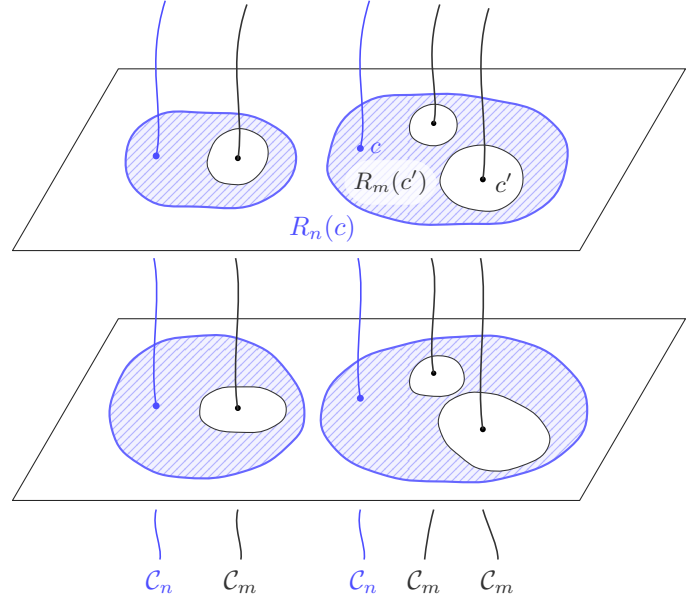

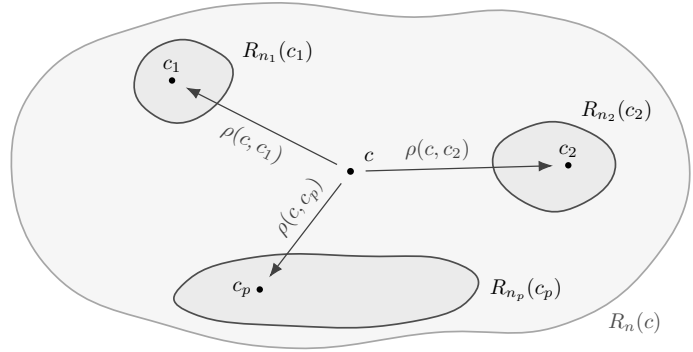
\begin{figure}[ht]
\centering
\scalebox{.8}{
\begin{tikzpicture}[scale=1.0,line join=round,
    reg/.style={thick,black!70,fill=gray!16},
    arr/.style={-{Latex[length=2.4mm]},semithick,black!75,
      shorten <=2.5mm,shorten >=2.5mm}]
  \draw[thick,gray!70,fill=gray!7]
    plot[smooth cycle,tension=0.72] coordinates {
      (5.6,0.15) (4.7,1.9) (2.9,2.62) (0.9,2.55) (-1.3,2.78)
      (-3.3,2.45) (-5.05,1.6) (-5.6,-0.05) (-4.85,-1.75) (-3.0,-2.55)
      (-0.9,-2.92) (1.4,-2.66) (3.45,-2.55) (4.98,-1.5) };
  \node[black!70,font=\small] at (4.7,-2.5) {$R_{n}(c)$};
  \draw[reg]
    plot[smooth cycle,tension=0.8] coordinates {
      (-1.95,1.55) (-2.35,2.05) (-3.05,2.12) (-3.55,1.6)
      (-3.35,1.0) (-2.6,0.85) };
  \fill (-2.95,1.5) circle (1.5pt) node[font=\small,anchor=south] {$c_1$};
  \node[font=\small] at (-1.2,1.98) {$R_{n_1}(c_1)$};
  \draw[reg]
    plot[smooth cycle,tension=0.82] coordinates {
      (2.35,0.15) (2.9,0.72) (3.78,0.72) (4.38,0.12)
      (3.8,-0.55) (2.95,-0.55) };
  \fill (3.6,0.1) circle (1.5pt) node[font=\small,anchor=south] {$c_2$};
  \node[font=\small] at (4.4,1.0) {$R_{n_2}(c_2)$};
  \draw[reg]
    plot[smooth cycle,tension=0.85] coordinates {
      (-2.75,-1.85) (-1.6,-1.42) (0.2,-1.4) (1.7,-1.5)
      (2.02,-2.05) (0.9,-2.5) (-1.0,-2.55) (-2.6,-2.35) };
  \fill (-1.5,-1.95) circle (1.5pt) node[anchor=east,font=\small] {$c_p$};
  \node[font=\small] at (2.9,-1.95) {$R_{n_p}(c_p)$};
  \draw[arr] (0.0,0.0) -- (-2.95,1.5)
        node[pos=0.5,sloped,below,font=\small] {$\rho(c,c_1)$};
  \draw[arr] (0.0,0.0) -- (3.6,0.1)
        node[pos=0.4,above,font=\small] {$\rho(c,c_2)$};
  \draw[arr] (0.0,0.0) -- (-1.5,-1.95)
        node[pos=0.4,sloped,above,font=\small] {$\rho(c,c_p)$};
  \fill (0.0,0.0) circle (1.6pt);
  \node[font=\small] at (0.3,0.24) {$c$};
\end{tikzpicture}
}
\caption{A region $R_{n}(c)$ with centre $c$ and the sub-regions
$R_{n_i}(c_i)$ reached from $c$ by the displacements $\rho(c,c_i)$.}
\label{fig:regions}
\end{figure}

The central notion in our work is that of a {\em Borel toast}, which is the descriptive set-theoretic
counterpart of Rokhlin nested towers in ergodic theory~\cite{OW}.

\begin{definition*}[Borel toasts]
{\rm Let $\bR^d \curvearrowright X$, $d\ge 2$,  be a free Borel action on a standard Borel space.
A} Borel toast {\rm is a sequence $(\cC_n, \la_n)_n$
of uniformly separated Borel cross-sections $\cC_n \subset X$ and Borel functions
$\la_n\colon \cC_n\to \cK$
satisfying the following conditions. For each $n$ and $c_n\in\cC_n$,
set $R_n(c_n)=T_{\lambda_n(c_n)} c_n$ (the {\em tile}).
Then:
\begin{itemize}
\item[(T1)] For any distinct $c_n, c_n'\in \cC_n$, the corresponding tiles are disjoint,
$R_n(c_n) \cap R_n(c_n') = \emptyset$.
\item[(T2)] For all $m<n$, $c_m\in \cC_m$, and $c_n\in \cC_n$, either $R_m(c_m) \cap R_n(c_n)=\emptyset$
or $R_m(c_m)\subset \operatorname{int} R_n(c_n)$.
\item[(T3)] For each $x\in X$ and each compact set $K\subset \bR^d$, there exist an $n$ and a $c_n\in \cC_n$
such that $T_K x \subset \operatorname{int} R_n(c_n)$.
\end{itemize}
Here and elsewhere, $\operatorname{int} R_n(c_n) = T_{\operatorname{int}\lambda_n(c_n)}c_n$.
}
\end{definition*}

For an illustration, see Figure~\ref{fig:toast}.

\subsubsection{Enhanced toasts}
In this work we will be using ``{\em enhanced Borel toasts}'' possessing additional properties.
The first one is a restriction on the shape of the compact sets $\la_n (c_n)$.
\begin{itemize}
\item[(T4)] The range of $\lambda_n$ belongs to the class $\cK_{*}$ of compact sets with non-empty
interior and connected complement, that is, for every $n$ and every $c_n\in\cC_n$, $\lambda_n(c_n)\in \cK_*$.
\end{itemize}
We call Borel toasts satisfying (T4) $\cK_*$-toasts.

\smallskip
The next condition will allow us to split each cross-section  $\cC_n$ into Borel pieces on which the data coming from lower levels is constant. We start with several definitions.

\medskip\noindent{\em Tagged cross-sections}: Set $\cC_{<n} = \bigcup_{m<n} \cC_m$.
A minor nuisance is that the cross-sections $\cC_m$ need not  be disjoint for different values
of $m$, so when discussing the lower-level points we regard  $\cC_{<n}$ as the disjoint union
\[
\widetilde \cC_{<n} = \bigsqcup_{m<n} ( \cC_m \times \{m\} )
\]
of tagged cross-sections.
Thus a lower-level point remembers its level. Often, we shall suppress the tag and continue to write
$c_m\in\cC_m$.

\medskip\noindent{\em Immediate predecessors\,}:
Fix $n$. A {\em predecessor} of $c_n\in\cC_n$ is a point $c_m\in\cC_{<n}$
such that $c_m\in R_n(c_n)$. We call a predecessor $c_m$ an {\em immediate predecessor}
of $c_n$ if it does not belong to any tile $R_k(c_k)$ with $m<k<n$.

Since the cross-sections are uniformly separated,
the total number $p=p_n(c_n)$ of immediate predecessors of any $c_n$ is finite.

\medskip\noindent{\em Borel linear order\,}:
We fix once and for all a {\em Borel linear order} on $X$, that is, a linear order on $X$, which is a Borel subset of $X\times X$ (it exists since, by the Kuratowski theorem, $X$ is Borel isomorphic to $\bR$).
It induces a Borel linear order on each $\cC_m$, and hence
a lexicographic Borel linear order on $\cC_{<n}$
\[
c_m < c'_{m'} \ \Longleftrightarrow (m<m')\ {\rm or\ } (m=m'\ {\rm and\ } c_m < c_{m'}').
\]

Let $\bar L_n(c_n) = \bigl(  c_{m_1}, \ldots , c_{m_p} \bigr)
\in (\cC_{<n})^p $ be the {\em list of immediate predecessors} of $c_n$
enumerated by the order on $\cC_{<n}$.
We also record their levels $\bar m(c_n) = (m_1, \ldots , m_p)$, and
the  corresponding shapes and positions of the tiles, setting
$\bar\lambda_n (c_n) = ( \lambda_{m_1} (c_{m_1}), \ldots, \lambda_{m_p} (c_{m_p}) )$
and $\bar\rho_n(c_n) =  \left( \rho(c_n, c_{m_1}), \ldots , \rho (c_n, c_{m_p}) \right)$.

\begin{itemize}
\item[(T5)]
For every $n$, there exists a countable Borel partition $\cC_n = \bigsqcup_\ell \cC_{n, \ell}$
such that on each $\cC_{n, \ell}$,
the number $p=p_n$ of immediate predecessors,
their levels $(m_1, \ldots , m_p)$,
their shapes $ ( \lambda_{m_1} (c_{m_1}), \ldots, \lambda_{m_p} (c_{m_p}) )$,
and their relative positions $ ( \rho(c_n, c_{m_1}), \ldots , \rho (c_n, c_{m_p}) )$
with respect to $c_n$
are constant.
\end{itemize}

\begin{definition*}[enhanced Borel toasts]
{\rm We call a Borel toast} enhanced {\rm if (in addition to (T1), (T2), and (T3)) it satisfies (T4) and (T5)}.
\end{definition*}

\begin{theorem}\label{thm:toasts}
For any free Borel action of $\bR^d$ on a standard Borel space, enhanced Borel toasts exist.
\end{theorem}

\subsection{Proof of Theorem~\ref{thm:toasts}}
The proof we present is borrowed from \cite[Theorem~5]{Slutsky2021}.

\subsubsection{Separation of smooth disks}
In the course of the proof of Theorem~\ref{thm:toasts} we will get a bit more than the theorem claims.
The range of the maps $\la_n$ will consist of  ``smooth disks'' in $\bR^d$, that is, the images of the closed unit
ball in $\bR^d$ under $C^\infty$-smooth diffeomorphisms. The proof employs the following
lemma~\cite[Lemma~3]{Slutsky2021}.

\begin{figure}[ht]
\centering
\scalebox{.9}{
\begin{tikzpicture}[scale=1.0,line join=round]
  \def\Ri{2.0}
  \def\Ro{3.4}
  \tikzset{smalldisk/.pic={
    \draw[thick,black!80,fill=black!14]
      plot[smooth cycle,tension=0.8] coordinates
        {(0:0.40) (60:0.34) (120:0.42) (180:0.32) (240:0.40) (300:0.30)};
  }}
  \draw[thick,gray!70,fill=gray!6] (0,0) circle (\Ro);
  \draw[thick,gray!70,dashed,fill=gray!12] (0,0) circle (\Ri);
  \draw[-{Latex[length=2mm]},thin] (0,0) -- (205:\Ri) node[pos=0.6,below,sloped,font=\small] {$R_1$};
  \draw[-{Latex[length=2mm]},thin] (0,0) -- (-35:\Ro) node[pos=0.8,below,sloped,font=\small] {$R_2$};
  \fill (0,0) circle (0.5pt);
  \draw[very thick,black!70,fill=black!8,fill opacity=0.45]
        plot[smooth cycle,tension=0.75] coordinates {
        (15:2.35) (55:2.55) (95:2.30) (135:2.6) (170:2.35)
        (205:2.55) (240:2.30) (275:2.6) (315:2.35) (345:2.55) };
  \node[black!70,font=\small] at (95:2.05) {$\mathscr D$};
  \foreach \x/\y/\rot/\sc in {0.25/0.35/20/1.0, -0.7/-0.2/-30/0.85, 0.55/-0.6/50/0.7, -0.35/0.85/0/0.6}
    {\pic[shift={(\x,\y)},rotate=\rot,scale=\sc] {smalldisk};}
  \node[black!80,font=\footnotesize] at (0.25,0.35) {$\mathscr D_i$};
  \foreach \x/\y/\rot/\sc in {2.85/0.5/15/0.55, -2.7/1.1/-20/0.5, -0.6/2.95/40/0.5, 1.4/-2.7/-10/0.55, -2.4/-1.85/25/0.5}
    {\pic[shift={(\x,\y)},rotate=\rot,scale=\sc] {smalldisk};}
  \node[black!80,font=\footnotesize] at (2.85,1.05) {$\mathscr D_j$};
\end{tikzpicture}
}
\caption{The separation lemma: a smooth disk $\mathscr D$ with $B_{R_1}\subset\mathscr D\subset B_{R_2}$ enclosing the inner disks $\mathscr D_i$ in its interior while staying disjoint from the outer disks $\mathscr D_j$.}
\label{fig:separation}
\end{figure}
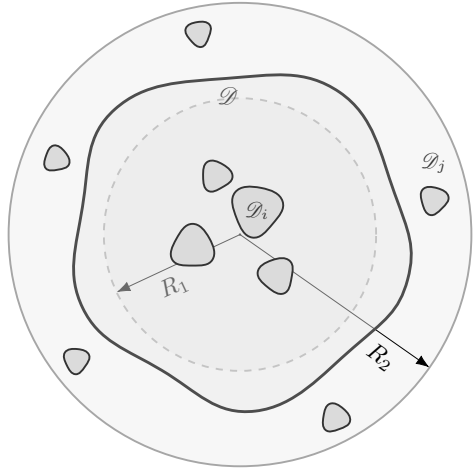

\begin{lemma}[separation lemma]\label{lemma:separation}
  Let $0<R_1<R_2$ and let $\mathscr D_1, \ldots, \mathscr D_p \subset \bR^d$, $d\ge 2$ be disjoint
  smooth disks of diameters less than $(R_2-R_1)/2$. Given a pair of concentric balls $B_{R_1}$ and
  $B_{R_2}$ of radii $R_1$ and $R_2$, there exists a smooth disk $\mathscr D$ such that
  $ B_{R_1} \subset \mathscr D \subset B_{R_2} $ and, for each $i$, either
  $\mathscr D_i \subset {\rm int}\mathscr D$ or $\mathscr D_i \cap \mathscr D = \emptyset$.
\end{lemma}

\subsubsection{Proof of Theorem~\ref{thm:toasts}}
Fix $r_n =6^n$,
and apply Theorem~\ref{thm:cross-sections} to the sequence $2r_n$.
Thus $\cC_n$ is $2r_n$-uniformly separated.
We need to construct the Borel maps $\la_n$ that map $\cC_n$ to smooth disks in $\bR^d$, satisfying
(T1), (T2), (T3), and (T5). We will do this by induction on $n$. In addition, we will require that
\begin{itemize}
\item[(T6)] For $n\ge 1$ and for any $c_n\in\cC_n$,
\[
B(r_{n-1}) \subset \la_n(c_n) \subset B(r_n).
\]
\end{itemize}
Here and elsewhere, by $B(r)$ we mean the ball in $\bR^d$ centred at the origin and of radius $r$.

\medskip
{\em The base} $n=1$ is immediate: we pick the constant function $\la_1(c_1) = \bar B(r_1/2)$,
$c_1\in\cC_1$.

\smallskip
{\em The induction step}: suppose that we have found Borel maps
$\la_i\colon \cC_i \to \cK_*$, $1\le i \le n-1$,
with range in smooth disks and such that $(\cC_i, \la_i)_{1\le i \le n-1}$ satisfies (T1), (T2), (T5),
and (T6).
Property (T3) will be verified after the induction procedure.

Fix $c_n\in\cC_n$, and consider the points
\[
c_{m_1}\in\cC_{m_1}, \ldots , c_{m_p}\in \cC_{m_p}, \
1 \le m_1 \le \ldots \le m_p \le n-1, \ p=p_n(c_n),
\]
such that the tiles $R_{m_1}(c_{m_1}), \ldots , R_{m_p}(c_{m_p})$
\begin{itemize}
\item
intersect the ball $T_{\bar B(r_n)} c_n $ centred at $c_n$
(i.e., $\la_{m_i}(c_{m_i}) + \rho(c_n, c_{m_i})$ hits the ball $\bar B(r_n)$);
\item and are not contained in a bigger such tile.
\end{itemize}
Note that since the cross-sections are uniformly separated, the number $p_n(c_n)$ of such points
is bounded uniformly in $c_n$.  Then we apply Lemma~\ref{lemma:separation} (the separation lemma)
to the smooth disks $\mathscr D_i = \la_{m_i}(c_{m_i}) + \rho(c_n, c_{m_i})$, $1\le i \le p$,
with $R_1 = r_{n-1}=6^{n-1}$,
$R_2 = r_n = 6^n$. The lemma provides us with a smooth disk $\la_n(c_n)=\mathscr D$
satisfying (T6) and such that every tile $R_{m_i}(c_{m_i})$ is either contained in the interior of the tile
$R_n(c_n) = T_{\la_n(c_n)} c_n$, or is disjoint from it.

This construction can be done so that only countably many distinct shapes for
$\la_n(c_n)= \mathscr D$ are used. Indeed, the input to Lemma~\ref{lemma:separation} which produces $\mathscr D$ is determined by the number $p$ of the tiles
intersecting $T_{\bar B(r_n)} c_n $, by the shapes of these tiles,
and by their locations relative to $c_n$.
By the inductive assumption, there are only countably many such tuples, and we can assume that
the same smooth disk $\mathscr D$ is used whenever the input tuple is the same. This will guarantee
countability of the number of possible shapes $\la_n(c_n)$.

To show that (T5) holds, we need another lemma.
\begin{lemma}\label{lemma:p_n} \mbox{}
Let $\la_m\colon\cC_m\to\cK$, $1\le m \le n-1$, be Borel maps satisfying conditions (T1) and  (T2),
and let $\la_n^*\colon\cC_n\to\cK$ be a Borel map such that the corresponding tiles $R_n^*(c_n)=T_{\la_n^*(c_n)}c_n$ are disjoint for distinct
$c_n\in\cC_n$. Then the number of immediate predecessors of each $c_n\in\cC_n$
inside $R_n^*(c_n)$  is a Borel function $p_n\colon \cC_n \to \bZ_{\ge 0}$, and
the associated finite data maps
\[
\bar m_n\colon \cC_n\to \bigsqcup_{p\ge 0} \{1, \ldots , n-1\}^p,
\quad \bar\lambda_n\colon \cC_n\to \bigsqcup_{p\ge 0}\cK^{p},
\quad \bar\rho_n\colon \cC_n\to \bigsqcup_{p\ge 0}\bigl( \bR^d \bigr)^{p},
\]
are Borel.
\end{lemma}
We proceed with the proof of
the inductive step in Theorem~\ref{thm:toasts}, relegating the
proof of the lemma to Appendix~B.

We apply Lemma~\ref{lemma:p_n} with $\la_n^*(c_n)=\bar B(2r_n)$. Note that by (T6)
if the tile $R_m(c_m)$, $m<n$, hits $T_{\bar B(r_n)}(c_n)$, then it is contained in the interior
of the tile $R_n^*(c_n)$, and therefore, the corresponding point $c_m$ is a predecessor of $c_n$.
By Lemma~\ref{lemma:p_n}, the new finite configuration data
\[
p, \quad (m_1, \ldots , m_p), \quad
(\lambda_{m_1}(c_{m_1}), \ldots , \lambda_{m_p}(c_{m_p})), \quad
(\rho (c_n, c_{m_1}), \ldots ,  \rho(c_n, c_{m_p}))
\]
depend on $c_n$ in a Borel way. By the induction hypothesis, the possible shapes $\lambda_m(c_m)$
form countable sets, and by Item 3 of Theorem~\ref{thm:cross-sections}, the relative positions
$\rho(c_n, c_m)$ belong to $\bQ^d$. Hence, this configuration map is countably-valued, and its level
sets define a countable Borel partition $\cC_n = \bigsqcup_\ell \cC_{n, \ell}$,
on whose elements the tuple which serves as an input to Lemma~\ref{lemma:separation} is constant. Hence, the output $\mathscr D = \la_n(c_n)$ of Lemma~\ref{lemma:separation}
is constant on each $\cC_{n, \ell}$, and therefore, the function $\la_n$ is Borel.

Applying Lemma~\ref{lemma:p_n} once again, this time with $\la_n^* = \la_n$, and refining,
whenever needed, the partition $\cC_n = \bigsqcup_\ell \cC_{n, \ell}$, we get property
(T5) at level $n$.

\medskip It only remains to verify Property (T3).
Fix a point $x\in X$ and a compact $K\subset \bR^d$. Then choose $m$ so that $K\subset B(r_m)$.
By the recurrence Property 2 of the sequence of cross-sections $(\cC_n)_n$ from Theorem~\ref{thm:cross-sections}, for any $\eps>0$, there
exists a sequence $(n_k)_{k}\subset\bN$ and points $c_{n_k}\in\cC_{n_k}$ such that
$|\rho (x, c_{n_k})|<2\eps r_{n_k}$. Take $n_k>m+1$. Then,
for every $y\in T_K x$,
\[
|\rho (c_{n_k}, y)| < 2\eps r_{n_k} + r_m < r_{n_k-1},
\]
provided that $\eps$ was taken sufficiently small. Recalling that the compact set $\la_{n_k}(c_{n_k})$ contains the ball $B(r_{n_k-1})$, we see that $T_K x \subset {\rm int\,} R_{n_k}(c_{n_k})$,
completing the proof of Theorem~\ref{thm:toasts}. \hfill $\Box$

\section{Proof of the measurable equivariant Weierstrass theorem}
Note that if the divisor $\Lambda$ is finite, then the product $\prod_{\la\in\Lambda}(z-\la)$
defines a Borel equivariant map from the set of finite divisors to the set of polynomials,
which is a right inverse to $Z$. Since both sets are Borel in $\cD_0$ and in $\cE$ respectively,
when proving Theorem~\ref{thm:main}, we will restrict ourselves to divisors with infinite support.

\medskip
Denote by $\cE^\times$ the space of non-vanishing entire functions.

\subsection{Preliminaries}

\subsubsection{A version of the Runge theorem}

\begin{claim}
\label{claim:Runge}
Let $K_{1},\ldots,K_{s}$ denote disjoint compact subsets of $\bC$, such that
$\bC\setminus \bigcup_{j\le s}K_j$ is connected.
Then, for any $\eps>0$ and any functions $F_1,\ldots,F_s\in\cE^\times$,
there exists a zero-free entire function $H$ such that
\[
\max_{1\le j \le s}\lVert H/F_j-1\rVert_{C(K_j)}\le \eps.
\]
\end{claim}

\begin{proof}
It suffices to prove the claim for $\eps<1$.
Let $F_j=e^{f_j}$, where $f_j$ are entire functions, $1\le j \le s$. Then, by the classical version of Runge's
theorem~\cite[XIII.1]{Gamelin}, choose the entire function $h$ so that
\[
\max_{1\le j \le s}\lVert h - f_j \rVert_{C(K_j)} <\tfrac13\, \eps,
\]
and set $H=e^h$. Then
\[
\lVert H/F_j-1 \rVert_{C(K_j)} = \lVert e^{h - f_j}-1 \rVert_{C(K_j)} < \eps,
\quad 1\le j \le s,
\]
since we assumed that  $\eps<1$.
\end{proof}

\subsubsection{The Weierstrass infinite product}
Here, we recall the standard Weierstrass infinite product construction $W\colon \cD\to \cE$
that supplies a Borel but not equivariant right inverse  to the divisor map $Z$.

First, we fix a Borel linear order on $\bC\setminus \{0\}$, letting
\[
z \prec w \Longleftrightarrow (|z|<|w|), \ {\rm or} \ (|z|=|w|, \ 0\le \arg(z)<\arg(w)<2\pi).
\]
Then we order the points of $\Lambda\setminus \{0\}$, $\Lambda\in\cD$,
according to this order, and enumerate the non-zero part $\Lambda\setminus\{0\}$
of the divisor with multiplicities
$ z_1 (\Lambda), z_2(\Lambda)$, \ldots , $z_j(\Lambda)$, \dots\, .
This enumeration  is Borel, i.e., the functions
\[
z_j\colon \cD \to \bC\setminus \{0\}, \quad j=1, 2, \ldots ,
\]
are Borel. For $z\ne 0$, we define the Weierstrass factor of genus $j$
\[
E_j(z) = (1-z) e^{z+z^2/2 + \ldots + z^j/j},
\]
and set
\[
W_\Lambda (z) = z^{m_0(\Lambda)} \prod_{j\ge 1} E_j \Bigl( \frac{z}{z_j(\Lambda)}\Bigr),
\]
where $m_0(\Lambda)$ is the multiplicity of the origin. The infinite product
converges locally uniformly (the use of genera $(j)$ is a convenient overkill).

For each $N$, the partial product
\[
W_{N, \Lambda} (z) = z^{m_0(\Lambda)} \prod_{1\le j\le N}
E_j \Bigl( \frac{z}{z_j(\Lambda)}\Bigr)
\]
is a Borel map $\cD\to\cE$,
because it depends on finitely many Borel-selected points and multiplicities, and the operations in
$\cE$ are continuous. Finally,
$W_{N, \Lambda}\to W_\Lambda$ in the Polish space $\cE$, and a pointwise limit of Borel maps
into a metric space is Borel. Hence $W$ is Borel.

\begin{remark*}[Remling~\cite{Remling}]
{\em Modifying the standard Weierstrass construction one can make the map
$W\colon \cD\to\cE$ continuous. We will not use this.
}
\end{remark*}

\subsection{Sequence of corrections}\label{subsect:corections}
Let $W_0$ be the Weierstrass map restricted to the non-periodic part $\cD_0\subset \cD$
of the divisor space.
The idea is to use a sequence of multiplicative corrections from $\cE^{\times}$ while
keeping Borelness, to make the map $W_0$ equivariant. This idea follows
quite closely Weiss' construction of measurable entire functions from \cite{Weiss}.

Let $(\cC_n, \la_n)_n$ be an enhanced Borel toast in $\cD_0$ satisfying conditions (T1)--(T4) and (T5),
and let $\mathcal{D}_n=\bigcup_{k\le n}\bigcup_{\Lambda\in\cC_k}R_k(\Lambda)$ be its
first $n$ layers. We aim to construct recursively a sequence of Borel maps
$W_n\colon \mathcal{D}_n\to \cE$, such that
\begin{enumerate}
\item[(a)] For $\Gamma\in\cD_n$, the entire function $W_{n, \Gamma}$ has the zero divisor
$\Gamma$.
\item[(b)] $W_{n,T_z \Gamma}=T_z W_{n,\Gamma}$ whenever the divisors $\Gamma$ and $T_z \Gamma$ belong to the same maximal tile in $\mathcal{D}_n$.

In particular, for $\Lambda\in\cC_n$ and for any of its immediate predecessors $\Lambda_m\in \cC_m$,
the entire functions $W_{n,\Lambda}$ and $T_{\rho(\Lambda_m ,\Lambda)} W_{m, \Lambda_m}$
have the same zero divisors.
\item[(c)] For $\Lambda\in\cC_n$, for its immediate predecessors
$\Lambda_m\in\cC_m$, $1\le m\le p$, and for the compact sets
$K_{\Lambda_m, \Lambda}  =\lambda_m(\Lambda_m)+\rho(\Lambda, \Lambda_m)$,
\[
\max_{1\le m \le p}\big\lVert W_{n,\Lambda}/T_{\rho(\Lambda_m, \Lambda)}W_{m, \Lambda_m}
-  1\big\rVert_{C(K_{\Lambda_m, \Lambda})} < 2^{-n}.
\]
\end{enumerate}

\noindent Item {\rm (a)} will ensure that $W_n$ is a right inverse to $Z$ on the
Borel sets $\mathcal{D}_n$ that exhaust $\cD_0$.
Item {\rm (c)} will ensure that the sequence of entire functions $(W_n)_{n}$ converges in $\cE$,
and {\rm (b)} amounts to a partial equivariance of $W_n$ on larger and larger domains $\cD_n$.

\medskip

We construct recursively a sequence of Borel maps
$H_n\colon \cD_n\to\cE^\times$ and
$W_n\colon \cD_n\to\cE$
as follows. We put first $H_0\equiv 1$ and recall that $W_0$ is the non-equivariant
measurable Weierstrass map.

\medskip

\noindent $\bullet$\;\; For $n\ge 1$, we first define maps $\Lambda\mapsto H_{n,\Lambda}$ for
$\Lambda\in\cC_n$.
Let, as above, $\bar L(\Lambda)=(\Lambda_{m_1}, \ldots , \Lambda_{m_p})$,
$\Lambda_{m_j}\in \cC_{m_j}$, be the list of immediate predecessors of $\Lambda\in\cC_n$
enumerated by the order on $\cC_{<n}$.
Denote by
\[
K_{\Lambda_m ,\Lambda} = \lambda_m(\Lambda_m) + \rho(\Lambda, \Lambda_m),
\quad  m\in \{m_1, \ldots , m_p\},
\]
the translation of the tile $\lambda_m(\Lambda_m)$, which superimposes
$\Lambda_m$ on $\Lambda$. We need the entire function $W_{n,\Lambda}$ to have
the same zero set as the previously defined entire function
$T_{\rho(\Lambda_m, \Lambda)}W_{m, \Lambda_m}$, and want these two functions to be
close to each other on $K_{\Lambda_m, \Lambda}$:
\begin{equation}\label{eq:quotient-bd}
\max_{m\in \{m_1, \ldots , m_p\}} \,
\big\lVert W_{n,\Lambda}/\bigl( T_{\rho(\Lambda_m, \Lambda)}W_{m,\Lambda_m} \bigr)
-1\big\rVert_{C(K_{\Lambda_m, \Lambda})} < 2^{-n}.
\end{equation}
The functions \(W_{n}\) and \(H_{n}\) will be related by the formula \(W_{n, \Lambda} = H_{n,
  \Lambda} W_{0, \Lambda}\).  In other words,  we are looking for $H_{n, \Lambda}\in\cE^\times$ such that
\begin{equation}\label{eq:quotient-bd-1}
\max_{m\in \{m_1, \ldots , m_p\}} \,
\big\lVert H_{n, \Lambda}/F_{\Lambda_m, \Lambda}
- 1 \big\rVert_{C(K_{\Lambda_m, \Lambda})} < 2^{-n},
\end{equation}
where
\begin{equation}\label{eq:prev-Fj}
F_{\Lambda_m, \Lambda} =
\frac{T_{\rho(\Lambda_m, \Lambda)}W_{m,\Lambda_m}}{W_{0,\Lambda}},
\quad m\in \{m_1, \ldots , m_p\}.
\end{equation}
Such a function $H_{n,\Lambda}$ exists.
Indeed, by construction of the enhanced Borel toast, the set
\[
\bC\setminus\bigcup_{j=1}^p K_{\Lambda_{m_{j}}, \Lambda}
\]
is connected, so Claim~\ref{claim:Runge} applies.
The measurability of $H_n$ will be proven below.

\medskip

\noindent $\bullet$\;\; We put $W_{n,\Lambda}=H_{n,\Lambda}W_{0,\Lambda}$ for
$\Lambda\in\cC_n$.
We extend $W_n$ to $\bigcup_{\Lambda\in\cC_n}R_n(\Lambda)$
by letting $W_{n,\Gamma}=T_{\rho(\Lambda, \Gamma)}W_{n,\Lambda}$,
$\Gamma\in R_n(\Lambda)$.
By the properties (T1) and (T2) of Borel toasts, there
are no obstructions due to collisions between different patches.
If for some $\Lambda'\in\cC_m$, $m<n$, we have that $R_{m}(\Lambda')$
is not a subset of some $R_k(\Lambda)$ with
$m<k\le n$, we let $W_{n, \Gamma}=W_{m, \Gamma}$ for $\Gamma\in R_{m}(\Lambda')$.
The function $W_{n,\Lambda}$ constructed that way is an entire
function with the zero divisor $\Lambda$, and with the partial equivariance (b).

\medskip

\noindent $\bullet$\;\; The limit $W= \lim_{n}W_n$ exists by virtue of \eqref{eq:quotient-bd}.
Indeed, for any fixed $\Gamma\in \cD_0$ and compact $K\subset\bC$, by the property (T3) of Borel
toasts, there exists $n_0$ such that $T_K\Gamma\subset R_{n_0}(\Lambda_0)$ for some
$\Lambda_0\in \cC_{n_0}$.  Moreover, $W_{n,\Gamma}$ are defined for $n\ge n_0$ and we have
$\left\lVert W_{n,\Gamma}/W_{n-1,\Gamma}-1\right\rVert_{C(K)}< 2^{-n}$ for each $n> n_0$.
Indeed,
if at the step $n$ the compact $T_K\Gamma$ is not swallowed by a new tile, then
$W_{n, \Gamma} = W_{n-1, \Gamma}$ there. If it is swallowed by a new tile, then by the nesting
property (T2) of Borel toasts, $T_K\Gamma$ is contained in one of the immediate predecessor tiles of
the new tile. Then estimate (c), transported by the equivariance, gives the desired bound on
$K$. Thus, writing
\[
W_{n,\Gamma} =
W_{n_0,\Gamma}\,\prod_{m=n_0+1}^{n}\frac{W_{m,\Gamma}}{W_{m-1,\Gamma}},
\]
the uniform convergence $W_{n,\Gamma}\to W_\Gamma$ on $K$ becomes clear.
The compact $K$ was arbitrary, so $W_\Gamma\in\cE$. Since the product on the RHS is zero-free,
it converges to a zero-free entire function, so $\Gamma$ is the zero divisor of $W_\Gamma$.

\subsection{Borelness}
It remains to explain how $H_{n}$ is chosen, and why this choice yields
Borelness of the map $W_n$.

\medskip

\noindent\underline{Step 0}.  Let $H_0\equiv 1$ and recall that $W_0$ is a measurable
right inverse to $Z$ on $\cD_0$. This completes the base step.

\medskip

\noindent \underline{Step $n\ge 1$}. We enter the $n$-th step in the following state:
we have constructed Borel maps $H_{n-1}$ and $W_{n-1}$
from $\cD_{n-1}$ (the union of first $n-1$ layers of the Borel toast)  to
$\cE^\times$ and $\cE$, in accordance
with (a), (b), and (c) in Section~\ref{subsect:corections}.
We proceed to define $H_n$.

\smallskip

We fix a piece of the Borel partition $(\cC_{n,\ell})_{\ell\ge 1}$ of $\cC_n$
from Property (T5), i.e., such that the data
\[
p, \quad (m_1, \dots , m_p), \quad (\lambda_{m_1}(\Lambda_1),\ldots,\lambda_{m_p}(\Lambda_p)), \quad
(\rho(\Lambda_1, \Lambda), \ldots, \rho(\Lambda_p, \Lambda))
\]
are constant for $\Lambda\in \cC_{n,\ell}$.

\smallskip

We first claim that, by the induction hypothesis, the
vector of non-vanishing entire functions
$(F_{\Lambda_1, \Lambda}, \ldots, F_{\Lambda_p,\Lambda})$ given by \eqref{eq:prev-Fj} is a Borel
function of $\Lambda$.
Indeed, we have
\[
F_{\Lambda_j, \Lambda}=T_{\rho(\Lambda_j, \Lambda)}W_{m_j,\Lambda_j}/ W_{0,\Lambda}.
\]
The Weierstrass map $W_{0,\Lambda}$ is Borel by construction, and
$W_{m_j, \Lambda_j}$ are Borel by the induction hypothesis. The entire functions
$T_{\rho(\Lambda_j, \Lambda)}W_{m_j,\Lambda_j}$ and $W_{0,\Lambda}$
have the same zero divisor, so their quotient is Borel by Borelness of the cocycle
for the action $\cE^\times \curvearrowright \cE$ (see~\eqref{eq:mult_action}).

\smallskip
Taking a dense sequence $(f_i)_{i\ge 1}$ in $\cE^\times$, define
\[
A_i = \bigl\{\Lambda\in \cC_{n, \ell}\colon
\max_{1\le j \le p} \| f_i/F_{\Lambda_{m_{j}}, \Lambda}-1\|_{C(K_{\Lambda_{m_{j}}, \Lambda})}
\le 2^{-(n+1)}
\bigr\}.
\]
On each $\cC_{n, \ell}$ the compacta $K_{\Lambda_{m_{j}}, \Lambda}$ written in the $\Lambda$-coordinate
system are fixed compact subsets $K_j\subset \bC$. Hence, Borelness of the sets $A_i$ follows by taking
suprema over fixed countable dense subsets of $K_j$'s, $1\le j \le p$.

By Claim~\ref{claim:Runge}, for each $\Lambda\in\cC_{n, \ell}$, there exists $H\in\cE^\times$
satisfying the needed inequality~\eqref{eq:quotient-bd-1}. Since $(f_i)_{i}$ is dense in $\cE^\times$,
some $f_i$ also satisfies this inequality. Therefore, $\cC_{n, \ell} = \bigcup_i A_i$. Put
\[
i_n(\Lambda) = \min\{ i\colon \Lambda\in A_i \}.
\]
Then
\[
\{ \Lambda\colon i_n(\Lambda)=i \} = A_i \setminus \bigcup_{r<i} A_r
\]
is Borel, and hence $H_{n,\Lambda}=f_{i_n(\Lambda)}$ is a Borel map
satisfying \eqref{eq:quotient-bd-1}.

Just as above, we define $W_n$ first for $\Lambda\in\cC_n$, and then extend it
by equivariance to tiles $R_n(\Lambda)$, keeping the old definition elsewhere.
To see Borelness of $W_n$, first we consider the union of new tiles
\[
\cD_n' = \bigcup_{\Lambda\in\cC_n} R_n(\Lambda)
\]
and extend $W_n$ on $\cD_n'$ by
\[
W_{n, \Gamma} = T_{\rho (b_n(\Gamma), \Gamma)} W_{n, b_n(\Gamma)},
\quad \Gamma\in \cD_n',
\]
where $b_n\colon \cD_n'\to \cC_n$ is the address map, which sends $\Gamma\in R_n(\Lambda)$
to $\Lambda$. This map is Borel by Claim~\ref{claim2} in Appendix~B. Hence,
$W_n$ is Borel on new tiles. On the remaining part of $\cD_{n-1}$ we keep the old definition,
$W_{n, \Gamma} = W_{n-1, \Gamma}$, $\Gamma\in \cD_{n-1}\setminus \cD_n'$.
Since $\cD_n = \cD_n' \cup \cD_{n-1}$ and both pieces are Borel,  $W_n$ is Borel on the
whole $\cD_n$. This completes the induction step.

\subsection{Completing the proof}
Now, $\cD_n \subset \cD_0$ and $W_n\colon \cD_n\to \cE$ are Borel. We extend $W_n$ to a Borel map
$\widetilde{W}_n\colon \cD_0\to\cE$, for instance, by letting $\widetilde{W}_n=W_0$ outside
$\cD_n$. Since $\cD_n\uparrow \cD_0$ by (T3), for each $\Lambda$ the sequence
$\widetilde{W}_{n, \Lambda}$ eventually agrees with the constructed sequence and
converges in $\cE$. Hence,
\[
W_\Lambda = \lim_{n\to\infty} \widetilde{W}_{n, \Lambda}
\]
is Borel as a pointwise limit of Borel maps into a Polish space.

\smallskip
The equivariance of the limiting $W$ is also straightforward. Given $\Lambda\in\cD_0$ and $w\in \bC$,
apply (T3) with a compact set in $\bC$ containing $0$ and $w$. Then, for some tile $R_n(c)$,
$\Lambda, T_w\Lambda\in R_n(c)$. From that level onward, the divisors $\Lambda$ and
$T_w \Lambda$ remain in the same
maximal cell in $\cD_r$, $r\ge n$, because of the condition (T2). Hence,
by Item (b) of our construction, $W_{r, T_w \Lambda} = T_w W_{r, \Lambda}$, $r\ge n$.  Passing to
the $r\to\infty$ limit yields the equivariance of the limiting map, $W_{T_w\Lambda} = T_w W_\Lambda$.
\hfill $\Box$

\section{Lack of measurable equivariant periodic inverses (proof of Theorem~\ref{thm:1-periodic})}

In the proof we will need one of the standard versions of the Phragm\'en--Lindel\"of maximum
principle~\cite[Section~7.3]{Marshall}:
\begin{lemma}\label{lemma:P-L}
Suppose that $f$ is a bounded holomorphic function in a horizontal strip $S$ and that
$|f|\le M$ on $\partial S$. Then $|f|\le M$ everywhere in $S$.
\end{lemma}

Denote by $\cE_1\subset \cE_{\sf per}$ and $\cD_1\subset \cD_{\sf per}$ the spaces of periodic entire functions and divisors with period $1$. Note that $\bZ$ acts continuously on $\cE_1$ and $\cD_1$ by
vertical translations: $(\tau_k f)(z)=f(z+{\rm i}k)$ and
$\tau_k \Lambda = \Lambda-{\rm i}k$, $k\in\bZ$.

\medskip
The following lemma will immediately yield Theorem~\ref{thm:1-periodic}:
\begin{lemma}\label{lemma:inv_meas}
The action $\bZ\curvearrowright \cE_1$ has no non-trivial invariant probability measures.
That is, every $\bZ$-invariant probability measure on $\cE_1$ is supported on the set of
constant functions.
\end{lemma}
\begin{proof}
Suppose $\bP$ is an invariant probability measure on $\cE_1$. We aim to show that
\[ \bP \{ f\ {\rm is\ constant\,} \} =1.\]
Let $I=[0, 1]\times \{0\}$ be the unit interval.
We claim that
\begin{itemize}
\item
{\em For $\bP$-a.e. $f\in\cE_1$, there exist $M>0$ and integer sequences $(s_n)_n$ and $(t_n)_n$
with
\[
\lim_{n\to\infty} s_n = \lim_{n\to\infty} t_n = +\infty,
\]
such that $|f|\le M$ holds both on $I-{\rm i}s_n$ and on $I+{\rm i}t_n$}.
\end{itemize}
This means that $\bP$-a.s.,  $|f|$ is bounded by $M$ on the horizontal lines ${\rm Im}(z)=-s_n$ and
${\rm Im}(z)=t_n$ for all $n$. Since $f$ is $1$-periodic, it is a priori bounded in any horizontal strip.
Then, by Lemma~\ref{lemma:P-L}, $|f|\le M$ in all the strips $\{-s_n \le {\rm Im}(z) \le t_n\}$.
Hence, it is bounded in $\bC$, and, by the Liouville theorem, $f$ is a constant function.

To prove the claim, for $M\in\bN$, we let $A_M = \{f \in \cE_{1}\colon  \sup_{I}|f| \le M\}$.
Clearly, $\cE_{1} = \bigcup_M A_M$. By the Poincar\'e recurrence lemma, for each $M$ and for
$\bP$-a.e. $f\in A_M$, there is a sequence of positive integers $t_k\uparrow +\infty$
and a sequence of negative integers $-s_k\downarrow -\infty$ such that $\tau_{t_{k}}f,
\tau_{-s_{k}}f \in A_M$.  This means that the absolute value of almost
every $f \in A_M$  is bounded by $M$ on both $I-{\rm i}s_{k}$ and $I+{\rm i}t_{k}$.
Hence, the claim.
\end{proof}

Theorem~\ref{thm:1-periodic} follows at once. Indeed, if a measurable equivariant Weierstrass map
$\sW\colon \cD_{\sf per}\to \cE_{\sf per}$ existed, then its restriction to $\cD_1$ would take
values in $\cE_1$. Now choose a non-trivial $\tau$-invariant probability measure\footnote{
For instance, let $\mathbb U$ be uniformly distributed on $[0, 1)$ and $\Pi$ be a stationary
point process on $\bR$ (for instance, a Poisson point process), $\mathbb U$ and $\Pi$ are independent,
and define a random measure on $\cD_1$ by
$ \Lambda = \{\mathbb U +n + {\rm i}t\colon n\in\bZ, t\in\Pi \}$.
} $\mu$ on $\cD_1$.
Then $\bP=\sW_*\mu$ is a $\bZ$-invariant probability measure on $\cE_1$.
By
Lemma~\ref{lemma:inv_meas}, $\bP$ is supported on constants. Since constants have
empty divisor and $Z\circ \sW=\mathrm{id}$, this implies that $\mu$ is supported on
the empty divisor, a contradiction.  \hfill $\Box$

\section{Lack of a measurable equivariant primitive
(proof of Theorem~\ref{thm:primitive})}\label{sect:primitive}

The proof combines Birkhoff's classical observation that the action $T\colon \bC\curvearrowright\cE$,
$(T_w F)(z)= F(z+w)$, has a dense orbit with the Baire category argument.

Recall that everywhere in this section $\cE$ denotes the space of all entire functions, including
the identically vanishing function.

\subsection{Prerequisites}

\subsubsection{Baire measurability}
Recall that a subset $A\subseteq X$ of a Polish space $X$ is {\em meager} (or of the first category)
if there are nowhere dense closed sets $F_n\subseteq X$, $n\in\bN$, such that $A\subseteq \bigcup_n F_n$.
The set $A$ is {\em comeager} (or residual), if its complement $X\setminus A$
is meager, that is, there exist dense open sets $U_n$, $n\in\bN$, such that $A\supseteq \bigcap_n U_n$.
By the Baire category theorem, every comeager set is dense in $X$.

\smallskip
The set $A\subseteq X$ is Baire-measurable (or has {\em the Baire property}),
if there exists an open set $U$ such that
$A\vartriangle U$ is meager, i.e., $A$ can be represented in the form $A=U\vartriangle M$, where
$U$ is open, and $M$ is meager. Note that all open sets and all meager sets have the Baire property.
Also note that the class of sets having the Baire property is a $\sigma$-algebra containing
all open sets~\cite[8.22]{Kechris}. In particular, all Borel sets are Baire-measurable.

For a Baire-measurable set $A$, put
\[
U(A)=\bigcup \bigl\{U\colon U\ {\rm open}, U\setminus A\ {\rm is\ meager} \bigr\}.
\]
Then,
\begin{enumerate}
\item[(a)] $A\vartriangle U(A)$ is meager;
\item[(b)] the set $U(A)$ is {\em regular}, i.e., it equals the interior of its closure;
\item[(c)] $U(A)$ is the unique regular open set $U$ such that $A\vartriangle U$ is meager.
\end{enumerate}
See~\cite[Theorem~8.29]{Kechris}.

\subsubsection{Generic ergodicity}
Let $G\curvearrowright X$ be a continuous action of a Polish group (a topological group whose topology
is Polish) $G$ on a Polish space $X$ ($\bC\curvearrowright \cE$ in our case). The action is called
{\em generically ergodic} if every $G$-invariant Baire-measurable set in $X$ is either meager or comeager.

The following classical lemma gives a convenient characterization of generic ergodicity.
\begin{lemma}\label{lemma:generic_ergod}
  Let $G\curvearrowright X$ be a continuous action of a Polish group $G$ on a Polish space $X$. Then
  generic ergodicity of the action is equivalent to its topological transitivity, i.e., to existence
  of a dense orbit.
\end{lemma}
Though we will use this lemma only in the direction
\[
{\sf existence\ of\ dense\ orbit} \Longrightarrow {\sf generic\ ergodicity},
\]
since the proofs are short, we will provide the proofs of both directions.

Recall that by the classical Birkhoff result~\cite{Birkhoff}, there exists an entire function $F$ with
a dense orbit $\{T_w F\colon w\in\bC\}$. Hence, {\em the action $T\colon \bC\curvearrowright \cE$ is generically ergodic}.

\medskip\noindent{\em Proof of Lemma~\ref{lemma:generic_ergod}}:
First, assume that the action is generically ergodic and fix a countable basis of non-empty open sets
$(U_n)_n$ in $X$. Then the set $\bigcap_n (G\cdot U_n)$ is $G$-invariant, and comeager, and therefore,
dense in $X$. Hence, the orbit of any point from this set is dense in $X$, proving topological transitivity.

\smallskip Now, assume that the action has a dense orbit, and let $A\subseteq X$ be a Baire-measurable,
$G$-invariant set. We need to show that it is either meager or comeager. We know that
$A=U\vartriangle M$, where $U$ is regular open, and $M$ is meager.
Observe that $U$ is also $G$-invariant. Indeed, for $g\in G$,
\[
U \vartriangle M  = A = gA = gU \vartriangle gM
\]
implies $U=gU$ since both $U$ and $gU$ are regular open, while $M$ and $gM$ are
both meager (two regular open sets whose symmetric difference is meager must coincide).
Since $U$ is open, any dense orbit should hit it. Since $U$ is $G$-invariant,
it contains the whole dense orbit, and therefore is dense (assuming that it is not empty);
since $U$ is regular open, this actually gives $U=X$.
Hence, either $U=\emptyset$ and $A$ is meager, or $U=X$, and $A$ is comeager.
\hfill $\Box$

\subsection{Proof of Theorem~\ref{thm:primitive}}
Let $\mathcal E$ be the Polish space of all entire functions, endowed with the
topology of locally uniform convergence. Suppose that there exists a Borel
equivariant right inverse $P\colon \mathcal E\to\mathcal E$ to the differentiation
operator $D$, that is, $D\circ P=\mathrm{id}$ and $P\circ T_w=T_w\circ P$.

Put
\[
\mathcal T=\{G\in\mathcal E:(P\circ D)G=G\}.
\]
Then $\mathcal T$ is Borel, and every differentiation fibre meets $\mathcal T$
in exactly one point. For each $F\in\mathcal E$, define
\[
\kappa(F)=F-(P\circ D)F\in\mathbb C.
\]
The map $\kappa:\mathcal E\to\mathbb C$ is Borel. Moreover,
\[
\kappa(T_wF)=\kappa(F),\qquad w\in\mathbb C,
\]
and
\[
\kappa(F+c)=\kappa(F)+c,\qquad c\in\mathbb C.
\]

Let $A=\{z:|z|\le 1\}$ and set
\[
\mathcal E'=\kappa^{-1}(A),\qquad
\mathcal E''=\kappa^{-1}(\mathbb C\setminus A).
\]
Then $\mathcal E'$ and $\mathcal E''$ are disjoint translation-invariant Borel sets
whose union is $\mathcal E$.

We claim that neither of these sets is meager. Since
\[
\mathbb C=\bigcup_{q\in\mathbb Z+{\rm i}\mathbb Z}(A+q),
\]
we have
\[
\mathcal E=\bigcup_{q\in\mathbb Z+{\rm i}\mathbb Z}(\mathcal E'+q).
\]
If $\mathcal E'$ were meager, then $\mathcal E$ would be meager, a contradiction.
The same argument shows that $\mathcal E''$ is not meager. Hence neither $\mathcal E'$ nor
$\mathcal E''$ is comeager.

This contradicts generic ergodicity of the translation action $\mathbb C\curvearrowright\mathcal E$.
\hfill $\Box$

\appendix

\numberwithin{theorem}{section}
\numberwithin{equation}{section}
\numberwithin{lemma}{section}
\numberwithin{claim}{section}
\numberwithin{corollary}{section}

\setcounter{theorem}{0}
\setcounter{lemma}{0}
\setcounter{claim}{0}
\setcounter{corollary}{0}

\section{Existence of a recurrent sequence of cross-sections
(proof of Theorem~\ref{thm:cross-sections})}
\label{app:rec-cross-sections}
We will provide a self-contained compilation of the proof of Theorem~\ref{thm:cross-sections}.  We
assume that $X$ is a standard Borel space and that $T\colon \bR^d \curvearrowright X$ is a free
Borel flow, which, whenever needed, we may assume continuous (as we have already mentioned, by the
Becker--Kechris theorem, the continuity is no restriction here).

Here is an outline of the proof.
\begin{itemize}
\item
Our starting point is a theorem of Kechris~\cite{Kechris1992, Kechris-Lectures},
which we accept without proof:
{\em For any free Borel action $\bR^d \curvearrowright X$,
there exists a Borel cross-section $\cC$.}

Note that Kechris' theorem holds in greater generality: the group $\bR^d$ may be replaced by an
arbitrary locally compact Polish group, and the action need not be free.

\item The first step is, given $r>0$, to make the cross-section $\cC$ $r$-separated.  The
  cross-sections constructed in Kechris' work~\cite{Kechris1992} are already
  \(r\)-separated. However, for the reader's convenience, we explain how an \(r\)-separated
  cross-section can be found.

\item
The second step is due to Kechris, Solecki, and Todorcevic~\cite{KST}.
Using a version of the greedy algorithm, they showed that retaining uniform
separation, one can fill the ``holes'' in the cross-section, achieving $2r$-density.

Note that almost the same argument yields $R$-density with any $R>r$.  We will not pursue this,
since we do not need that precision; for our purposes $2r$-density is as good as $100r$-density.

\item
The next step is due to Boykin and Jackson~\cite{BJ}.
They proved existence of a sequence $(\cC_i)_i$ of $r_i$-separated and $2r_i$-dense
cross-sections, $r_i\to\infty$, with the following recurrence property (property~2 in Theorem~\ref{thm:cross-sections}): {\em for every $x\in X$ and
every $\eps>0$, there are infinitely many $i$ for which}
$\displaystyle \inf_{c_i\in\cC_i} |\rho(x, c_i)| < \eps r_i$.

\item The last step is {\em the rationality} property of the sequence of cross-sections $(\cC_i)_i$,
  that is, property~3 in Theorem~\ref{thm:cross-sections}. It is due to
  Slutsky~\cite[Lemma~2.3]{Slutsky2017}.  The main difficulty of that step is the proof of existence
  of {\em a rational grid} for the action $T$, that is, {\em a Borel set $Y\subset X$ which is
    invariant under the action of $\bQ^d$ and which intersects each orbit of the flow in a unique
    $\bQ^d$-orbit}. In turn, this construction is based on (a special case of) the hyperfiniteness
  of the equivalence relation $E_{\cC}$ established by Jackson, Kechris, and Louveau~\cite{JKL} in the context
  of actions of groups of polynomial growth. Here $E_\cC$ is the restriction of the orbit
  equivalence relation $E_T$ to the cross-section $\cC$. The hyperfiniteness means that the
  equivalence relation $E_\cC$ is an increasing union of finite Borel equivalence relations.
\end{itemize}

\subsection{Uniformly separated cross-sections}
\label{subsect:unif-separated}

Having a Borel cross-section $\cC \subset X$, we need to thin it out, i.e., given $r>0$, to find
an $r$-separated Borel cross-section $\cC_r \subset \cC$. For this, it will be convenient to use the
language of Borel graphs.  Recall that, for a standard Borel space $X$, {\em a Borel graph} is a
Borel irreflexive and symmetric relation $G\subset X\times X$ (i.e., $(x,x)\notin G$ for any
$x\in X$, and $(x, y)\in G$, $x, y\in X$ yields $(y, x)\in G$).

We say that a Borel graph is
\emph{locally countable} if each vertex has only countably many neighbours.  This class of graphs is
particularly easy to work with since for any Borel \(A \subseteq X\), the collection of neighbours of
\(A\), i.e., the set \[\mathcal N(A)=\{y \in X : (x,y) \in G \text{ for some } x \in A\}\]
is Borel.  Indeed, since each vertex of $G$ has countably many neighbours, for each $y\in\mathcal N(A)$,
the sections $\{x\in A\colon (x, y)\in G\}$ are countable, so by the Luzin--Novikov theorem, the set
$\mathcal N(A)$, which is the projection of $\{(x, y)\in G\colon x\in A\}$ on the second coordinate
is Borel.

A subset $J\subseteq X$ is called {\em independent} if no two distinct points of it are connected by
an edge. An independent set $J$ is {\em maximal}, if it is maximal by inclusion.
We shall use a standard greedy construction of maximal independent sets
for locally countable Borel graphs whose vertices can be covered by countably many Borel independent sets.

\begin{lemma}[maximal independent sets]\label{lemma:indep_sets_independent_cover}
Let $X$ be a standard Borel space and let $G\subset X\times X$ be a Borel graph. Assume that
$G$ is locally countable. If \(X\) can be
covered by a countable family of Borel independent sets, then each Borel independent set $J\subseteq X$
can be extended to a maximal Borel independent set.
\end{lemma}

\begin{proof}
  Let \((J_{n})_{n}\) be a countable family of Borel independent sets such that
  \(X = \bigcup_{n} J_{n}\).  We can assume that $J_0=J$.
A maximal independent set will be constructed using a
  ``greedy-algorithm'': starting with \(J_{0}\), we inductively add elements of \(J_{n}\) that have
  no edges to the points picked at the previous stages of the construction.  More precisely, set
  \(I_{0} = J_{0}\) and, once $I_0, \ldots , I_{n-1}$ have been constructed, put
  $I_{<n}=I_0 \cup \ldots \cup I_{n-1}$, and set
\[
I_n = \{x\in J_n\colon x\ {\rm is\ not\ adjacent\ to\ any\ point\ of\ } I_{<n} \}.
\]
Each $I_n$ is Borel, since, as noted above, the set of neighbours of a Borel set is Borel.  Note
that \(\bigcup_{k \le n} I_{k}\) is a Borel independent set.

At last, set
\[
I = \bigcup_{n\ge 0} I_n.
\]
Clearly, $I$ is Borel and independent as an increasing union of independent sets
\(\bigcup_{k \le n}I_{k}\).  It is also maximal. Indeed, take $x\notin I$, and let \(n\) be such
that \(x \in J_{n}\).  If $x\notin I_n$, then, by definition of $I_n$, $x$ is adjacent to some point
of $I_{<n}\subset I$, so we cannot add $x$ to $I$ without destroying independence. Thus, $I$ is a
Borel maximal independent set containing $J$, proving the lemma.
\end{proof}

The Borel graph $G$ is {\em locally finite} if each of its vertices has only finitely many neighbours.
For locally finite Borel graphs, the assumption that \(X\) can be
covered by a countable family of Borel independent sets is redundant.
\begin{corollary}[maximal independent sets]\label{cor:indep_sets}
Let $X$ be a standard Borel space and let $G\subset X\times X$ be a Borel graph. Assume that
$G$ is locally finite. Then $G$ admits a Borel maximal independent set of vertices.
\end{corollary}

\begin{proof}
  Since all uncountable Borel spaces are isomorphic, and the case of a countable \(X\) is trivial,
  we may assume without loss of generality that \(X = 2^{\bN}\) is the Cantor space.  Consider the
  countable collection $\mathscr Z = (Z_n)_n$ of cylindrical subsets of \(X\), i.e., sets of the
  form
  \[[s] = \{x \in 2^{\bN}\colon  x \text{ extends } s\},\]
  where \(s\) is a finite binary string.  Adjoining the empty set to \(\mathscr Z\), we get a
  countable family of Borel sets that separates points (i.e., for any pair $x\ne x'$ in $X$, there
  exists $\cZ\in\mathscr Z$ such that $x\in \cZ$, while $x'\notin \cZ$) and is closed under finite
  intersections.

  Observe that, for any finite subset $S=\{x_1, \ldots , x_N\}$ of $X$ and any $x\in X\setminus S$,
  there exists $Z\in \mathscr Z$ such that $x\in Z$, while $S\subset X\setminus Z$.  Indeed, for
  each $1\le i \le N$, take $n_i$ so that $x\in Z_{n_i}$, while $x_i\notin Z_{n_i}$, and set
  $Z = \bigcap_{i=1}^N Z_{n_i}$. Since $\mathscr Z$ is closed under finite intersections,
  $Z\in\mathscr Z$.

  For each $x\in X$, the set of its neighbours $\mathcal N(x)=\{y\colon (x, y)\in G\}$ is finite, so we can
  find $Z_n\in\mathscr Z$ that separates $x$ from $\mathcal N(x)$, i.e., such that $x\in Z_n$, while
  $\mathcal N(x)\cap Z_n = \emptyset$. Put
  \[
  I_{n} \stackrel{\rm def}= \{x \in X\colon
n \text{ is minimal such that } x \in Z_{n} \text{ and } \mathcal N(x)\cap Z_n = \emptyset\}.
  \]
  This is an independent subset of $G$, and the family \((I_{n})_{n}\) covers \(X\).  Furthermore,
  these sets are Borel.  Indeed, Theorem~\ref{thm:L-N-ordered} (the ordered finite-section version of
the Luzin--Novikov theorem) gives us Borel uniformizations
  \(g_{k}\colon  X \to X\), \(k \in \bN\), that enumerate neighbours of \(x\). Then
  \[  I_{n} = \{x \in X\colon n \text{ is minimal such that } x \in Z_{n} \text{ and }
    g_{k}(x) \not \in Z_{n} \text{ for all }k \text{ satisfying } g_{k}(x) \ne x\}\]
are Borel, and
Lemma~\ref{lemma:indep_sets_independent_cover} does the job.
\end{proof}

Corollary~\ref{cor:indep_sets} gives us the needed thinning of $\cC$.  More precisely, define a
graph $G_r \subset \cC\times \cC$ by
\[
(c, c')\in G_r \quad \Longleftrightarrow  \quad 0 < | \rho(c,c') | \le 2r.
\]
This is a Borel graph on $\cC$. A Borel maximal independent set $\cC_r\subset \cC$ is then
$r$-separated. Maximality ensures that $\cC_r$ still meets every orbit: every point
of $\cC$ is either selected or is adjacent to a selected point.

\subsection{Uniformly separated and relatively dense cross-sections}
\label{subsect:rel-dense}
Now, we will fill the holes in the $r$-separated cross-section $\cC$, extending it to a $2r$-dense
set, while keeping $r$-separation.

First, we define the rational saturation of $\cC$, letting
\[
Y = T_{\bQ^d}\, \cC = \bigcup_{q\in\bQ^d} T_q\, \cC.
\]
This set is Borel. Indeed, for each fixed $q$, the map $T_q\colon X\to X$ is a Borel bijection with
Borel inverse $T_{-q}$, so by the Luzin--Suslin theorem, $T_q\, \cC$ is Borel, and therefore,
$Y$ is also Borel.  Note that on each orbit, $Y$ is countable and dense.

As in the previous step, we define a graph $G$ on $Y$ by
\[
(y, y')\in G \quad \Longleftrightarrow \quad 0 < | \rho(y, y') | \le 2r.
\]
This graph is Borel and locally countable. Since $\cC$ was $r$-separated, each $T_{q}\mathcal{C}$ is
independent in $G$. Lemma~\ref{lemma:indep_sets_independent_cover} guarantees existence of a Borel
maximal independent extension  \(\cD\supseteq \cC\). Clearly, independence yields $r$-separation of points of $\cD$. It
remains to observe that maximality yields its $2r$-relative density.  Indeed, by maximality of $\cD$
in $Y$, for any $y\in Y$, there is $d\in \cD$ on the same orbit, such that
$| \rho (y, d) |\le 2r$. Since $Y$ was dense on each orbit, $| \rho (x, d) |\le 2r$ holds for any point
$x\in X$. \mbox{} \hfill $\Box$

\subsection{Recurrent sequences of cross-sections}
\noindent The next result provides us with a sequence of $r_i$-separated and $2r_i$-dense
cross-sections, $r_i\to\infty$, which have good recurrence properties.

\begin{lemma}[Boykin--Jackson]
\label{thm:Boykin-Jackson}
Let $\bR^d\curvearrowright X$ be a free Borel flow
on a standard Borel space $X$, and let
$1\le r_1<r_2<\ldots$ be a sequence of positive
reals tending to infinity. Then, there exists a sequence
of  $r_i$-separated and $2r_i$-dense Borel cross-sections $(\cC_i)_i$, such that,
for every $x\in X$ and $\eps>0$,  we have
\[
\inf_{c\in\cC_i} |\rho(x, c)| \le \eps r_i
\]
for infinitely many $i\ge 1$.
\end{lemma}

In addition to the original proof in~\cite{BJ}, a variant of this idea has been used by Marks and
Unger in~\cite[Lemma~A.2]{Marks-Unger}.
Both results were stated  for Borel actions $\bZ^d\curvearrowright  X$ while we need the
corresponding statement for Borel flows $\bR^d\curvearrowright X$.
For that reason, we reproduce the argument from~\cite{Marks-Unger} in full detail.

\begin{proof}
We already know that for each $i\ge 1$, there exists an $r_i$-separated and
$2r_i$-dense cross-section $\cD_i$. To obtain the desired cross-section $\cC_i$, we will
shift the points of $\cD_i$ appropriately to ensure that the required recurrence holds.
Let $x\in X$. As long as the points $d\in \cD_i$ from the orbit of $x$
are  shifted by the same $w=f_{i}(x)\in\bR^d$, $r_i$-separation and $2r_i$-density are preserved.

\smallskip
For $x\in X$, we denote by $D_i(x)$ ``the picture'' of $\cD_i$ on the orbit of $x$,
viewed from $x$ put at the origin, and scaled by $r_i$:
\begin{equation}
\label{eq:def-Di}
D_i(x)=\Big\{\frac{\rho(x, d)}{r_i}\colon
d\in \cD_i,\; d E_T x\Big\}.
\end{equation}
Denote by $D^*(x)$ the upper limit of the sets $D_i(x)$, that is,
\[
D^*(x) = \bigcap_{j\ge 1} \overline{\bigcup_{i\ge j} D_i(x)}.
\]
This set enjoys the following properties:
\begin{enumerate}
\item[(1)] For each $x\in X$, $D^*(x)\cap \bar B(2)$  is a non-empty compact set
($\bar B(s)$ denotes the closed ball in $\bR^d$ of radius $s$
centred at the origin).
\item[(2)] The map $x \mapsto D^*(x)$ is constant on the orbits of the flow, i.e., $D^*(x)=D^*(y)$, whenever $x E_T y$.
\end{enumerate}
Indeed, the property (1) follows from $2r_i$-density of $\cD_i$. Property (2) follows from the fact that if
$x, x'\in X$ lie on the same orbit, then $D_i(x)=D_i(x') + r_i^{-1} \rho(x, x')$,
so the sets $D_i(x)$ and $D_i(x')$ have the
same upper limits by virtue of $r_i\to\infty$.

\smallskip
Note that if $0\in D^*(x)$, then the desired recurrence holds on the whole
orbit of $x$, and no shift would be needed there. Indeed, there exists a subsequence
$(i_k)_{k\ge 1}$ of indices and points $d_{i_k}\in \cD_{i_k}$ such that
$|\rho(x, d_{i_k})| /r_{i_k}\to 0$ as $k\to\infty$.

\smallskip
To construct the shifts we will make use of a Borel map $f\colon X\to\bR^d$ which assigns a
point of $D^*(x)$ to each orbit. That is, we ask that $f$ is $\bR^d$-invariant
and that $f(x)\in D^*(x)$ for every $x\in X$. To find such a map, we take $f(x)$ to be
the \emph{lexicographically least} element of $D^*(x)\cap \bar B(2)$ (note that this set is non-empty
by item 1 above). By the $\bR^d$-invariance of the set $D^*$, the map $f$ is also $\bR^d$-invariant, and $f(x)\in D^*(x)$ by construction. Borelness of $f$ follows from a combination of the following two claims:
\begin{claim}\label{claim:BorelnessEffros}
The map $x\mapsto D^*(x)\cap \bar B(2)$ is Borel from $X$ to the Effros Borel space
$\cK(\bar B(2))$ of non-empty compact subsets of $\bar B(2)$.
\end{claim}
\begin{claim}\label{claim:lex_least}
Let $B\subset \bR^d$ be an open ball.
The lexicographically least selector is a Borel map from the Effros Borel space
$\cK(\bar B)$ to $\bR^d$.
\end{claim}

\noindent
We proceed with the proof of Lemma~\ref{thm:Boykin-Jackson}, and then return
to the proofs of Claims~\ref{claim:BorelnessEffros} and~\ref{claim:lex_least}.

\medskip

Since $D^*(x)$ is the upper limit of the sets $D_i(x)$ and since
$f(x)\in D^*(x)$, it follows that for every $x\in X$ and $\eps>0$, there are infinitely many
$i$ for which $f(x)$ is $\eps$-close to $D_i(x)$,
\begin{equation}
\label{eq:dist-f-Di}
\inf_{w\in D_i(x)}|f(x)-w|< \eps,
\end{equation}
that is,  $|r_i f(x)-r_i w|\le \eps r_i$. But by the definition~\eqref{eq:def-Di} of $D_i(x)$,
there exists $d_i\in \cD_i$  with $r_i w=\rho(x, d_i)$. Thus,
\[
|\rho (x, T_{- r_i f(x)} d_i)| = |\rho (x, d_i) - r_i f(x)|<\eps r_i,
\]
that is, the shifted point $c_i = T_{-r_i f(x)} d_i$ is $\eps r_i$-close to $x$.
Letting $ \cC_i = \{ T_{-r_i f(d)}\,d\colon d\in\cD_i\}$,
we complete the proof of Lemma~\ref{thm:Boykin-Jackson} (modulo
the proofs of Claims~\ref{claim:BorelnessEffros} and~\ref{claim:lex_least}).
\end{proof}

\noindent{\em Proof of Claim~\ref{claim:BorelnessEffros}}:
First, observe that, for an open set $O\subset \bR^d$, the set
$\{x\colon D_i(x) \cap O \ne \emptyset \}$ is Borel, as the projection
of the Borel set
\[
\{ (x, d)\in E_T \cap (X\times \cD_i)\colon \rho (x, d)/r_i \in O \}
\]
on the first coordinate (recall that $E_T\subset X\times X$ is the orbit equivalence relation generated
by the flow $T$). The sections over $x$ are countable because $\cD_i$ is a cross-section, so
the Luzin--Novikov theorem applies.

The accumulation compact set $K(x)=D^*(x)\cap \bar B(2)$ can be described by countably many
such tests. Take any relatively open subset of $\bar B(2)$, that is, $U=\bar B(2)\cap O$,
where $O$ is open, and take a countable collection $\mathcal V$ of open sets $V\subset \bR^d$,
which is a base of the Euclidean topology. Then $K(x)\cap U \ne \emptyset$ is equivalent to
\[
\exists V\in\mathcal V\ {\rm with\ } \bar V\subset O \ {\rm s.t.\ }
\forall j\ge 1\; \forall N \; \exists i>N \; V\cap D_i(x)\cap B(2+1/j)\ne \emptyset.
\]
This gives a Borel expression since it is built from countable unions and intersections of
the Borel sets \[ \{x\in X\colon V\cap D_i(x)\cap B(2+1/j)\ne \emptyset \}, \]
and proves that the compact-valued map $x\mapsto D^*(x)\cap\bar B(2)$ is Effros--Borel.
\mbox{} \hfill $\Box$

\medskip\noindent{\em Proof of Claim~\ref{claim:lex_least}}:
Let $K \in \mathcal K(\bar B)$ be a non-empty compact set. Put
\[
m_1(K) = \min \{w_1\colon (w_1, \ldots , w_d)\in K\}.
\]
The map $\mathcal K(\bar B)\ni K\mapsto m_1(K) \in \bR$ is Borel. To see this, we need
to check that, for every real $t$, the set $\{K\in\mathcal K(\bar B)\colon m_1(K)<t\}$ is
Borel. But $m_1(K)<t$ is equivalent to $K\cap \{w\colon w_1<t\}\ne \emptyset$, while the
set $ \{w\colon w_1<t\} \cap \bar B$ is relatively open in $\bar B$. Therefore,
$K\cap \{w\colon w_1<t\}\ne \emptyset$ is Borel.

Next, put
\[
m_2(K)   = \min \{w_2\colon (m_1(K), w_2, \ldots , w_d)\in K\}.
\]
Then
\begin{multline*}
\{K\colon m_2(K)<t\} \\
= \bigcup_{\substack {q\in \bQ \\ q<t}}\ \bigcap_{j\ge 1}\ \bigcup_{p\in\bQ}
\bigl( \{K\colon p - 1/j < m_1(K) < p \} \cap \{K\colon K\cap\{w\colon w_1<p, w_2<q \}\ne \emptyset\}
\bigr).
\end{multline*}
Hence, the map $m_2$ is Borel.

We continue inductively. Having constructed Borel functions $m_1(K), \ldots , m_{r-1}(K)$,
we define
\[
m_r(K) = \min\{w_r \colon w\in K, w_1=m_1(K), \ldots , w_{r-1} = m_{r-1}(K)\},
\]
and, using the same argument, show that the map $m_r$ is Borel. At the end,
\[ m (K) = (m_1(K), \ldots , m_d(K)) \] is the lexicographically least point of $K$, and the map
$K\mapsto m(K)$ is Borel.  \mbox{}\hfill $\Box$

\subsection{Hyperfiniteness of cross-section equivalence relations}
\label{s:hyperfinite}
Recall that a Borel equivalence relation $E$ is hyperfinite
if it is an increasing union of finite Borel equivalence relations.
\begin{lemma}[Jackson--Kechris--Louveau~\cite{JKL}]
\label{thm:hyperfinite}
Let $\cC$ be a uniformly separated Borel cross-section, and denote by $E_{\cC}$ the restriction of the
orbit equivalence relation to $\cC$. Then $E_\cC$ is hyperfinite.

Furthermore, let $E_\cC= \bigcup_{i\ge 1} E_i$ be an increasing union of
finite Borel equivalence relations.
Then there exists a Borel strong {\rm (}that is, irreflexive{\rm )}
partial order $\prec$ on $\cC$ such that distinct \(c, c' \in \cC\) are \(\prec\)-comparable if and only if
\(c E_{\cC} c'\) and all \(E_{i}\)-classes are intervals; that is, if $a\prec b$, $b\prec c$, and $a\,E_i\,c$,
then $a\,E_i\,b$, as well.
\end{lemma}
The proof we present follows the ideas from Boykin--Jackson~\cite{BJ} and
Marks--Unger~\cite[Lemma~A.2]{Marks-Unger}.  It uses existence of a sequence of uniformly separated
and relatively dense cross-sections with the Boykin--Jackson recurrence property, which was constructed in the
previous sections.

\begin{proof}
By the Boykin--Jackson theorem applied with $r_k=5^k$, there are
Borel cross-sections $\cC_k$ which are $r_k=5^k$-separated
and $2r_k = 2\cdot 5^k$-dense, such that for any fixed $\eps>0$ and any
$x\in X$, there are infinitely many indices $k$ for which
\begin{equation}
\label{eq:recurrence2}
| \rho(x, c_k)| \le \eps \, 5^k,
\end{equation}
for some $c_k\in\cC_k$.  We use the cross-sections $(\cC_k)_{k\ge 1}$ to group the elements
of $\cC$ in increasingly large but finite groups. These groups will
be the equivalence classes for the equivalence relations $E_k$, $k\ge 1$.

Fix a Borel linear order on $X$.

\smallskip We let $\pi_1\colon \cC\to \cC_1$ be a Borel map that assigns
to each $c\in\cC$ one of the closest to $c$ points from $\cC_1$ lying on the same orbit
(the closest point exists since $\cC_1$ is relatively dense and locally finite on each orbit).
To define $\pi_1$, consider the Borel set
\[
\{(c, c_1)\in \cC\times \cC_1 \colon c_1E_Tc, \, |\rho(c, c_1)|
= \min \{|\rho(c, c_1')|\colon c_1'\in\cC_1, \, c_1' E_T c\}\,
\}.
\]
Its sections over the first coordinate are finite and non-empty. So, the ordered
finite-section version of the Luzin--Novikov theorem (Theorem~\ref{thm:L-N-ordered})
uniformizes this set by the closest point (least in the fixed Borel order on $X$ at ties),
giving us a Borel $\pi_1$.

We define $E_1$ by declaring that its equivalence classes are the inverse images
\[
E_1(c_1)\overset{\rm def}{=}\pi^{-1}_1(c_1)
\]
indexed by $c_1\in\cC_1$. Note that $\cC_1$ may be entirely
disjoint from $\cC$, so despite appearances we do not ask that
$c_1$ is an element of the class $E_1(c_1)$.
By the relative density of $\cC_1$ and local finiteness of $\cC$, each equivalence class is finite.
Also note that whenever $|\rho(c, c')|<5$ for some $c'\in \cC_1$, we have
$c\in E_1(c')$. Moreover, each $c\in\cC$ lies in
$E_1(c_1)$ for some $c_1\in\cC_1$ with $|\rho(c, c_1)|\le 10$.

\smallskip
Fix $k\ge 1$, and assume that $E_{k}$ has been constructed
by assigning equivalence classes $E_{k}(c_k)$, indexed by $c_k\in\cC_{k}$.
Suppose further that the following properties hold:
\begin{enumerate}
\item[(i)] If $c \in\cC$ is such that $|\rho(c, c')|<2 \cdot 5^{k-1}$ for
some $c'\in\cC_{k}$, then $c \in E_{k}(c')$.

\item[(ii)] Every $c \in\cC$ lies in $E_{k}(c')$ for some
$c'\in\cC_{k}$ with $|\rho(c, c')|\le 3\cdot 5^{k}$.
\end{enumerate}

We proceed to construct $E_{k+1}$ with this induction hypothesis retained.
Just as above, using the ordered finite-section version of the Luzin--Novikov theorem,
we construct the Borel map $\pi_{k+1}\colon \cC_{k}\to\cC_{k+1}$
that associates to $c_k\in\cC_{k}$ the closest point
in $\{c_{k+1}\in\cC_{k+1}\colon c_{k+1}E_T c_k\}$. We define a Borel equivalence relation
$E_{k+1}$ by de\-clar\-ing that its equivalence classes are
\[
E_{k+1}(c_{k+1})\overset{\rm def}{=}\bigcup_{c'\in \pi_{k+1}^{-1}(c_{k+1})}E_{k}(c'),
\qquad c_{k+1}\in \cC_{k+1}.
\]
It is again evident that $E_{k+1}$ is Borel and has finite equivalence classes.

\smallskip\noindent {\sf Property (i) holds for $k\mapsto k+1$}. Fix $c\in\cC$
and suppose that for some $c'\in\cC_{k+1}$ we have
$| \rho(c, c') |<2 \cdot 5^{k}$.
The point $c$ belongs to $E_{k}(c'')$ for some $c''\in\cC_{k}$ such that
$|\rho(c, c'')| \le 3\cdot 5^{k}$. But then $|\rho(c', c'')| \le (2+3)\cdot 5^{k}=5^{k+1}$.
Since $\cC_{k+1}$ is $5^{k+1}$-separated, this implies
that $\pi_{k+1}(c'')=c'$. Thus the whole equivalence class $E_{k}(c'')$ is paired
with $c'\in \cC_{k+1}$, so $c\in E_{k+1}(c')$. It follows that
the induction hypothesis (i) holds with $k$ replaced by $k+1$.

\smallskip\noindent {\sf Property (ii) holds for $k\mapsto k+1$}.
By the induction hypothesis (ii), any $c\in \cC$ lies in $E_{k}(c')$ for some $c'\in\cC_{k}$ with
$|\rho(c, c')|\le 3\cdot 5^{k}$.
But $\cC_{k+1}$ is $2\cdot 5^{k+1}$-dense, so the closest point
$c''=\pi_{k+1}(c')$ to $c'$ satisfies $|\rho(c', c'')| \le 2\cdot 5^{k+1}$,
whence
\[
| \rho(c, c'') |\le 3\cdot 5^{k}+2\cdot 5^{k+1}< 3\cdot 5^{k+1}.
\]
This means that $c\in E_{k}(c')\subset E_{k+1}(c'')$, so also
the induction hypothesis (ii) holds with $k+1$ in place of $k$.

\smallskip This completes
the inductive process, which provides us with an increasing
sequence of equivalence relations $(E_k)$
with properties (i) and (ii) guaranteed for all $k\ge 1$.

\smallskip
It only remains to verify that $E_{\cC}=\bigcup_{k\ge 1}E_k$.
But this follows from item (i) combined with the recurrence property
\eqref{eq:recurrence2}. Indeed, fix any two points $c, c'\in\cC$ with
$c\, E_\cC\, c'$, and fix $\eps$ with $0<\eps<2/5$.
By the recurrence \eqref{eq:recurrence2},
there are infinitely many indices $i$ for which there exists $c_i\in\cC_i$
with $| \rho(c, c_i)| \le \eps 5^i < 2\cdot 5^{i-1}$. By item (i) above, this implies that
$c\in E_{i}(c_i)$ for these indices. But $| \rho(c, c')|$
is a fixed finite number, so for $i$ sufficiently large, we have
\[
| \rho(c', c_{i}) | \le |\rho(c, c')| + |\rho(c, c_i)| \le |\rho(c, c')| +\eps 5^{i} < 2\cdot 5^{i-1}.
\]
Hence, again by item (i), we have $c'\in E_{i}(c_{i})$ for the same indices
(starting from some $i_0$).
But then $c\, E_{i}\, c'$ for some $i$, which completes the proof.

\medskip It remains to prove the existence of a Borel partial order $\prec$ on $\cC$ whose restriction to
each $E_\cC$-class is linear. The partial order on $\cC$, which we will construct, will be {\em strict} (i.e., irreflexive).

Fix any global strong Borel linear order $<$ on $\cC$, and define an order $\prec_i$ on each $E_i$ inductively. For $i=1$, set
\[
x \prec_1 y \ \Longleftrightarrow \ x\, E_1\, y, \ x<y.
\]

Assume $\prec_i$ has been defined. For $c\in\cC$, set $s_i(c)  = \displaystyle \min [\, c\, ]_{E_i}$,
where $ [\, c\, ]_{E_i}$ is the $E_i$-equivalence class of $c$, and the minimum is taken with respect
to the linear order $<$\,. Since the $E_i$ classes are finite,
by the Luzin--Novikov theorem, $s_i\colon \cC\to\cC$ is a Borel map. Now, define $\prec_{i+1}$ on
$E_{i+1}$-classes, by
\[
x \prec_{i+1} y \quad \Longleftrightarrow \quad
x\, E_{i+1}\, y \ \ \text{and\ } \bigl(\, (s_i(x)<s_i(y)) \ \text{or}\  (s_i(x)=s_i(y)
\ \text{and\ }x\prec_i y)\, \bigr).
\]
This relation is Borel and $\prec_{i+1}$ extends $\prec_i$.

Finally, for $x, y \in\cC$, set
\[
x \prec y \quad  \Longleftrightarrow \quad \exists i\ x\prec_i y.
\]
Since $\prec_i$'s are compatible, this is a well-defined strict Borel partial order on $\cC$.
If $x\, E_{\cC}\, y$ and $x\ne y$, then $x\, E_i\, y$ for some $i$, so $x$ and $y$ are
$\prec$-comparable. If $x\prec y$ and $y\prec z$, choose $i$ large enough so that $x \prec_i y$
and $y \prec_i z$, then
$x\prec_i z$, and therefore $x\prec z$. Hence, $\prec$ is linear on each $E_{\cC}$-class.  \end{proof}

\subsection{Existence of a rational grid}
\noindent Here, we prove Lemma 2.3 from~\cite{Slutsky2017}.
\begin{definition*}[rational grid]
A Borel subset $\Sigma\subset X$ is called a rational grid if
\begin{enumerate}
\item $\Sigma$ is invariant under $\bQ^d$.
\item $\Sigma$ intersects every $\bR^d$-orbit.
\item For every $\sigma\in\Sigma$, $\Sigma\cap [\sigma]_{\bR^d} = [\sigma]_{\bQ^d}$.
\end{enumerate}
\end{definition*}
\noindent ($[\sigma]_{\bR^d}$ and $[\sigma]_{\bQ^d}$ denote the $\bR^d$- and $\bQ^d$-orbits of $\sigma$).
Note that equivalently, the last property says that, for every
$\sigma, \sigma'\in\Sigma$ lying on the same $\bR^d$-orbit, $\rho(\sigma, \sigma')\in \bQ^d$.
\begin{lemma}
\label{lemma:rational-grid}
Any free Borel flow $\bR^d\curvearrowright X$ admits a rational grid.
\end{lemma}
The proof goes as follows. We start with a $10$-separated Borel cross-section $\cC_0$ and
construct inductively a convergent sequence of small perturbations
\[
\cC_0 \xrightarrow{\varphi_0} \cC_1 \xrightarrow{\varphi_1}\cC_2 \xrightarrow{\varphi_2} \ldots \,,
\]
where $\varphi_i$ are Borel injective mappings, which move points along their orbits (hence,
each $\cC_i$ remains a Borel cross-section).
The induction will be tuned so that, for the limiting cross-section $\cC$, its $\bQ^d$-orbit
$\Sigma = [\cC]_{\bQ^d}$ becomes a rational grid.

\begin{proof}
Fix a $10$-separated Borel cross-section $\cC_0$, and let $E_{\cC_0}$
be its orbit equivalence relation. Lemma~\ref{thm:hyperfinite} gives us
an increasing sequence of finite equivalence relations $E_n$ such that
$E_{\cC_0}=\bigcup_n E_n$, and a Borel strong partial order $\prec$
on $\cC_0$ for which $E_n$-equivalence classes are intervals.

\smallskip
By a \emph{spiral of cross-sections} we mean a sequence $\cC_n$ of Borel cross-sections,
together with Borel maps $h_n\colon \cC_n\to B(2^{-n-1})\subset \bR^d$ such that
\[
\cC_{n+1}
\overset{\rm def}{=}\{ T_{h_n(c_n)}\,c_n\colon c_n\in\cC_{n}\}.
\]
We let \[ \varphi_{n}\colon \cC_n\ni c_n\mapsto T_{h_n(c_n)}\, c_n\in \cC_{n+1}. \]
For indices $0\le k< n$, we put
\[
\varphi_{k, n}=\varphi_{n-1}\circ\ldots\circ\varphi_{k}, \quad \varphi_{k, n}\colon \cC_k\to\cC_n,
\]
and let $\varphi_{n,n}={\rm id}_{\cC_n}$.
If $(\cC_n,h_n)_n$ is a spiral, then for $c_0\in\cC_0$,
\[
\rho(c_0, \varphi_{0,n}(c_0))=\sum_{j=0}^{n-1} h_j(\varphi_{0,j}(c_0)).
\]
Since the range of  $h_n$ is contained in $B(2^{-n-1})$,
the spiral  $(\cC_n,h_n)_n$ {\em converges}, that is, for each $c_0\in\cC_0$, the above sum
converges as $n\to\infty$. We define the limiting maps
$H\colon \cC_0\to B(1)\subset \bR^d$ by
\[
H(c_0)=\sum_{j\ge 0} h_j(\varphi_{0,j}(c_0)).
\]
Since $h_j$ and $\varphi_{0,j}$ are Borel, so is $H$.
The limiting cross-section of a convergent spiral
$(\cC_n,h_n)_n$ is defined as
\[
\cC
\overset{{\rm def}}{=} \{T_{H(c_0)}\, c_0\colon c_0\in\cC_0\}.
\]
The limiting map $ \varphi\colon c_0 \mapsto T_{H(c_0)} c_0$ is Borel.
Since we chose $\cC_0$ to be $10$-separated, while the length of the limiting
perturbation cannot exceed $1$, $\varphi$ is injective. Hence, by the Luzin--Suslin theorem,
$\cC$ is Borel. Since $\varphi$ only moves the points along their orbits, $\cC$ is a cross-section.

\smallskip
We denote by $E_k^n$ the equivalence relation $E_k$ pushed forward from $\cC_0$ to
$\cC_n$ by $\varphi_{0, n}$; that is,
\[
c_n\, E_k^n \,c_n'\Longleftrightarrow
\varphi^{-1}_{0,n}(c_n)\, E_k\, \varphi^{-1}_{0,n}(c_n').
\]
We adopt the convention that $E_{-1}^0$ is the trivial equivalence relation, i.e.,
its equivalence classes are singletons.

\smallskip
In what follows, we will inductively  construct a spiral $(\cC_n, h_n)_{n\ge 0}$
with the following properties:
\begin{enumerate}
\item[(1)] $h_n$ is constant on $E_{n-1}^{n}$-equivalence classes; i.e.,
$ c_n\, E_{n-1}^n\, c_n' \Longrightarrow h_n(c_n)=h_n(c_n')$.
\item[(2)] Each $E_{n}^{n+1}$-class is ``on a rational grid''; i.e.,
$ c_{n+1}\, E_{n}^{n+1}\, c_{n+1}' \Longrightarrow \rho(c_{n+1}, c_{n+1}')\in\bQ^d$.
\end{enumerate}
For $\eps>0$, we let $\alpha_\eps$ be a Borel map
$\alpha_\eps: \bR^d\to\bQ^d$, such that $|\alpha_\eps(x)-x| < \eps$.

\medskip\noindent {\sf Base step.}
In the first generation we denote by $s_0\colon \cC_0\to\cC_0$
the Borel selector, which associates to each $c$ the minimal element
in its $E_0$-class with respect to the order $\prec$.
We set
\[
h_0(c_0)=\alpha_{1/2}(\rho(s_0(c_0), c_0)) - \rho(s_0(c_0), c_0),
\quad
\varphi_0 (c_0) = T_{h_0(c_0)}\, c_0,
\quad
\cC_1 = \varphi_0(\cC_0).
\]
Since $E_{-1}^0$ is trivial, there is nothing to respect for property (1).
To check (2), note that, for $c_0\, E_0\, c_0'$, we have $s_0(c_0)=s_0(c_0')$,
which we denote by $s_0$. Then,
\begin{align*}
\rho (\varphi_0(c_0), \varphi_0(c_0')) &= \rho (\varphi_0(c_0), c_0) +
\rho (c_0, c_0') + \rho(c_0', \varphi_0(c_0')) \\
&= - h_0(c_0) + \rho (c_0, c_0') + h_0(c_0') \\
&= -\alpha_{1/2}(\rho(s_0, c_0)) + \rho(s_0, c_0) + \rho (c_0, c_0')
+ \alpha_{1/2}(\rho(s_0, c_0')) + \rho(c_0', s_0) \\
&= -\alpha_{1/2}(\rho(s_0, c_0)) + \alpha_{1/2}(\rho(s_0, c_0')).
\end{align*}
The RHS belongs to $\bQ^d$, so (2) also holds.

\medskip\noindent
{\sf Inductive step}.
Suppose now that $(\cC_m)_{m\le n}$ and $(h_m)_{m\le {n-1}}$ have been constructed
so that properties (1) and (2) hold for all indices $0\le m<n$.
We proceed to construct $h_n\colon \cC_{n}\to B(2^{-n-1})$ as follows.
We let $s_n\colon \cC_n\to\cC_n$ be the Borel selector which picks for each $c_n\in\cC_n$
the $\prec$-minimal element\footnote{
We transport the order $\prec$ from $\cC_0$ to $\cC_n$ by $\varphi_{0, n}$, letting
\[
c_n \prec c_{n}' \Longleftrightarrow
\varphi_{0, n}^{-1} (c_n) \prec \varphi_{0, n}^{-1} (c_n').
\]
}
of its $E_{n}^n$-class, and $t_n\colon \cC_n\to\cC_n$ be
the Borel selector which picks for each $c_n\in\cC_n$ the $\prec$-minimal element of its
$E_{n-1}^n$-class (note that on the base step, the $E_{-1}^0$ equivalence classes were
singleton, so we used $t_0=\text{id}_{\cC_0}$). We put
\begin{align*}
h_n(c_n)=\alpha_{2^{-n-1}}\big(\rho(s_{n}(c_n), t_{n}(c_n))\big) &- \rho(s_{n}(c_n),t_{n}(c_n)), \\
\varphi_n (c_{n}) &= T_{h_n(c_n)}\, c_n, \quad \cC_{n+1} = \varphi_n(\cC_n).
\end{align*}
Since $t_n$ is constant on $E_{n-1}^n$-classes, and $s_n$ is constant on $E_n^n$-classes
and $E_n$ increase with $n$, $h_n$ is constant on $E_{n-1}^n$-classes. That is,
property (1) holds.

To check (2), suppose that
$c_{n+1}=\varphi_n(c_n)$, $c_{n+1}'=\varphi_n(c_n')$, where $c_n$ and $c_n'$ are
in the same $E_n^n$-equivalence class. Then
\begin{multline*}
\label{eq:rho-complicated}
\rho(c_{n+1}, c_{n+1}') =\rho (\varphi_{n}(c_n), \varphi_{n}(c_n') ) \\
=\rho (\varphi_n(c_n), \varphi_n(t_n(c_n)) )
+ \rho (\varphi_n(t_n(c_n)), \varphi_n(t_n(c_n')))
+\rho (\varphi_n(t_n(c_n')), \varphi_n(c_n')).
\end{multline*}
We claim that each of the three terms on the right-hand side is rational.
First, the points $c_n$ and $t_n(c_n)$ are in the same $E_{n-1}^n$-equivalence class,
so $\rho (c_n, t_n(c_n)) \in \bQ^d$ by the induction assumption (2). Since $\varphi_n$
acts by the same translation on these two points, we see that
 \[ \rho (\varphi_n(c_n), \varphi_n(t_n(c_n)) ) = \rho (c_n, t_n(c_n)) \in \bQ^d. \]
For the same reason, $ \rho (\varphi_n(t_n(c_n')), \varphi_n(c_n')) \in\bQ^d$.
At last,
\begin{align}\label{eq:rho-long}
\nonumber \rho (\varphi_n(t_n(c_n)), \varphi_n(t_n(c_n'))) &=
\rho (\varphi_n (t_n(c_n)), t_n(c_n)) + \rho (t_n(c_n), t_n(c_n'))
+ \rho(t_n(c_n'), \varphi_n (t_n(c_n'))) \\
&= -h_n(t_n(c_n)) + \rho (t_n(c_n), t_n(c_n')) + h_n(t_n(c_n')) .
\end{align}
Noting that, since $c_n$ and $c_n'$ are in the same $E_n^n$-class,
$t_n(c_n)$ and $t_n(c_n')$ are also in the same class, and
\[ s_n(t_n(c_n))=s_n(c_n)=s_n(t_n(c_n'))=s_n(c_n')=s_n, \]
we see that
\begin{align*}
\text{RHS\ of\ }\eqref{eq:rho-long} &=
-\alpha_{2^{-n-1}}(\rho (s_n, t_n(c_n))) + \rho (s_n, t_n(c_n))
+ \rho (t_n(c_n), t_n(c_n'))  \\
&\qquad + \alpha_{2^{-n-1}}(\rho (s_n, t_n(c_n'))) - \rho (s_n, t_n(c_n')) \\
&= -\alpha_{2^{-n-1}}(\rho (s_n, t_n(c_n)))  + \alpha_{2^{-n-1}}(\rho (s_n, t_n(c_n'))),
\end{align*}
and also lies in $\bQ^d$. Thus, property (2) also holds.
This completes the induction.

\medskip
Now, let $\cC = \{T_{H(c_0)}\, c_0\colon c_0\in\cC_0\}$ be the limiting Borel cross-section.
Define $\Sigma = [\cC]_{\bQ^d}$. It is a $\bQ^d$-invariant Borel set, which intersects every
$\bR^d$-orbit. It remains to check that all points of $\cC$ lying on the same $\bR^d$-orbit
are rational translates of each other.
Indeed, let $c, c'\in\cC$ lie on the same orbit. Then
\[
c= T_{H(c_0)}\, c_0 = \lim_n \varphi_{0, n}(c_0), \quad
c'=T_{H(c_0')}\, c_0'   = \lim_n \varphi_{0, n}(c_0'),
\]
with $c_0, c_0'\in\cC_0$ also lying
on the same orbit. Then $c_0$ and $c_0'$ are $E_N$-equivalent for some $N$. Then,
at stage $N$ Property (2) gives that
\[ \rho (\varphi_{0, N+1}(c_0), \varphi_{0, N+1}(c_0'))\in\bQ^d. \]
By Property (1), for all later stages, the perturbations are constant on the relevant older class,
so this rationality is preserved. Hence, in the limit, $\rho(c, c')\in\bQ^d$, completing the proof
that $\Sigma$ is a rational grid.
\end{proof}

\subsection{Rational cross-sections}
The last step towards Theorem~\ref{thm:cross-sections}
is the following lemma:
\begin{lemma}
\label{lemma:rational}
The sequence $(\cC_i)_{i\ge 1}$ of uniformly separated and relatively dense Borel
cross-sections with the recurrence property from Lemma~\ref{thm:Boykin-Jackson}
can be taken to be rational, i.e., satisfying property 3 of Theorem~\ref{thm:cross-sections}:
for every $m$ and $n$ and $c_m\in\cC_m$ and $c_n\in\cC_n$ lying on the same orbit,
$\rho(c_m, c_n)\in\bQ^d$.
\end{lemma}

\begin{proof}
Having at hand the rational grid $\Sigma$ from Lemma~\ref{lemma:rational-grid}, we start with the
sequence $(\cC_i')_{i\ge 1}$ from Lemma~\ref{thm:Boykin-Jackson}
with parameters $r'_i=r_i+1$, and proceed to select
cross-sections $\cC_i\subset \Sigma$ and Borel
bijections $\xi_i\colon \cC_i'\to\cC_i$
such that for any $i\ge 1$ and $c_i\in\cC_i'$
\begin{enumerate}
\item $\xi_i(c_i)$ lies on the same orbit as $c_i$,
\item $|\rho(c_i, \xi_i(c_i))|<1$.
\end{enumerate}
This step needs a short explanation. For each $i$ consider the relation
\[
R_i = \{(c, \sigma)\in\cC_i' \times \Sigma\colon \sigma\in [c]_{\bR^d}, \;
|\rho(c, \sigma)| <1\}.
\]
It can be represented in the form
$ R_i = (\cC_i' \times \Sigma) \cap E_T \cap \rho^{-1}(B(1)) $,
and hence, is Borel. Since $\Sigma\cap [c]_{\bR^d} =  [\sigma]_{\bQ^d}$,
it is dense and countable. Hence, each vertical section of $R_i$ is non-empty and
countable. Thus, by the Luzin--Novikov theorem, we get a uniformizing Borel map
$\xi_i\colon \cC_i'\to \Sigma$ such that $(c, \xi_i(c))\in R_i$ for all $c\in \cC_i'$.
Then $\xi_i$ is injective\footnote{
Since $\cC_i'$ is $(r_i+1)$-separated, for $c\ne c'\in \cC_i'$, the
$1$-neighbourhoods of these points are disjoint, whence, $\xi_i(c)\ne \xi_i(c')$.},
so, by the Luzin--Suslin theorem, $\cC_i=\xi_i(\cC_i')$ is Borel.

\smallskip
A minor nuisance is that the cross-sections $\cC_i = \xi_i(\cC_i')\subset \Sigma$
are now $r_i$-separated, recurrent, but only $(2r_i+3)$-dense. So we need to repeat the
``filling in holes'' construction from Section~\ref{subsect:rel-dense}, preserving the
rationality (enlarging a cross-section preserves recurrence).
By construction in Section~\ref{subsect:rel-dense},
the new cross-sections $\cD_i$ are chosen inside the same rational grid $\Sigma$,
which coincides with the rational saturation $[\cC_i]_{\bQ^d}$ of $\cC_i$.
Thus, rationality of $(\cD_i)_i$ is preserved. \end{proof}

\section{Proof of Lemma~\ref{lemma:p_n} }\label{AppB}
We will be using the following  version of the Luzin--Novikov theorem.
\begin{theorem}[ordered finite-section form of the Luzin--Novikov theorem]\label{thm:L-N-ordered}
Let $X$ and $Y$ be standard Borel spaces, let $<$ be a Borel linear order on $Y$, and let
$B\subseteq X\times Y$ be Borel with finite sections
\[
B_x=\{y\in Y\colon (x,y)\in B\}.
\]
Then the counting function
\[
N_B\colon X\to\bN \cup \{0\},\qquad N_B(x)=|B_x|,
\]
is Borel.

Moreover, for every $j\ge1$, there is a Borel map
\[
e_j\colon \{x\in X\colon N_B(x)\ge j\}\to Y
\]
such that, whenever $N_B(x)=p$, the section $B_x$ is listed in increasing order as
\[
B_x=\{e_1(x)<e_2(x)<\ldots<e_p(x)\}.
\]
In particular, for each $p$, the map
\[
x\mapsto (e_1(x),\ldots,e_p(x))
\]
is Borel on the Borel set $\{x\colon N_B(x)=p\}$.
\end{theorem}
\begin{proof}
By the Luzin--Novikov theorem, we write $B=\bigsqcup_{r\ge1} B_r$, where each $B_r$ is the graph of a Borel map
$h_r\colon X_r\to Y$ with Borel domain $X_r\subseteq X$. Then
\[
N_B(x)=\sum_{r\ge1}\done_{X_r}(x),
\]
so $N_B$ is Borel as an $\bN\cup\{0,\infty\}$-valued function; since all sections are finite, it in fact takes values in $\bN\cup\{0\}$.

It remains to obtain the ordered enumeration. We first select the least element of each non-empty section.
The set of pairs $(x,y)\in B$ for which some smaller point of the same section exists is the projection to
$X\times Y$ of the Borel set
\[
\{(x,y,z)\colon (x,y)\in B,\ (x,z)\in B,\ z<y\}.
\]
The sections of this set over $(x,y)$ are countable, so this projection is Borel by the Luzin--Novikov theorem.
Hence the graph
\[
G_1=\{(x,y)\in B\colon y\text{ is the }<\text{-least element of }B_x\}
\]
is Borel. Since $G_1$ is the graph of a function on the Borel set $\{N_B\ge1\}$, this function is Borel; call it $e_1$.

Assume that $e_1,\ldots,e_j$ have been constructed. On the Borel set $\{N_B\ge j+1\}$, apply the preceding least-element argument to the Borel relation
\[
B^{(j+1)}=\{(x,y)\in B\colon N_B(x)\ge j+1\text{ and }e_i(x)<y\text{ for }1\le i\le j\}.
\]
This gives a Borel map $e_{j+1}$ selecting the least element of the remaining tail of $B_x$. Induction gives the claimed increasing enumeration.
\end{proof}

\begin{claim}[the address map]\label{claim2}
Let $X_n = \bigsqcup_{c_n\in\cC_n} R_n(c_n)$ and let
$b_n\colon X_n\to \cC_n$ be the function that maps $x\in R_n(c_n)$ to its unique representative
$c_n$. Then the function $b_n$ is Borel. In particular, $X_n$ is Borel.
\end{claim}

\noindent{\em Proof of Claim~\ref{claim2}}: First, we note that the set
\[ \{(c_n, \la_n(c_n))\in \cC_n \times \cK\colon c_n\in\cC_n\}\] is Borel as a graph of a Borel
function~\cite[Theorem~12.4]{Kechris}, and that the set
\[ \cC_{n, \la_n} = \{(c_n, w)\in \cC_n\times\bR^d\colon w\in\la_n(c_n)\}\] is also Borel as the zero set
of a Borel function
\[
\cC_n \times \bR^d \ni (c_n, w) \mapsto \operatorname{dist}_{\bR^d} (w, \la_n(c_n)) \in\bR_{\ge 0}.
\]
Furthermore, the map
\[
\Phi_n\colon \cC_{n, \la_n} \ni (c_n, w) \mapsto T_w c_n \in X_n
\]
is Borel (since the action $T$ is Borel), injective (if $T_w c_n = T_{w'}c_n'$ with $w\in\la_n(c_n)$,
$w'\in \la_n(c_n')$, then $T_w c_n$ belongs to both tiles $R_n(c_n)$ and $R_n(c_n')$; hence,
$c_n=c_n'$, and since the action is free, $w'=w$), and is surjective. Therefore, by the Luzin--Suslin theorem,
the set $X_n$ is Borel, and
\[
\Phi_n^{-1}\colon X_n \ni x \mapsto (c_n, \rho(c_n, x))\in \cC_{n, \la_n}
\]
 is also Borel. Note that $b_n = \pi_{\cC_n} \circ \Phi_n^{-1}$, where $\pi_{\cC_n}$ is the
projection of $\cC_n\times \bR^d$ on the first coordinate, which obviously is continuous. Hence, the function
$b_n$ is Borel. \hfill $\Box$

\medskip\noindent{\em Proof of Lemma~\ref{lemma:p_n}}:
For each $1\le k\le n$, let
\[
X_k=\bigsqcup_{c_k\in\cC_k}R_k(c_k)
\]
and let $b_k\colon X_k\to\cC_k$ be the address map from Claim~\ref{claim2}. We extend it to all of $X$ by adding
$\infty$ to $\cC_k$ as an isolated point and setting
\[
b_k^*(x) =
\begin{cases}
b_k(x), & x\in X_k, \\
\infty, & x\in X\setminus X_k.
\end{cases}
\]
Each map $b_k^*\colon X\to\cC_k\cup\{\infty\}$ is Borel.

For $m<n$, set
\begin{align*}
A_{m, n} &= \{(c_m, c_n): b_n^*(c_m)=c_n\}, \\
B_{m, k} &= \{(c_m, c_n)\colon b_k^*(c_m)=\infty\},\qquad m<k<n,
\end{align*}
and define
\begin{align*}
I_{m, n} &=
\{(c_m, c_n)\colon b_n^*(c_m)=c_n
\ {\rm and\ } b_k^*(c_m)=\infty\ {\rm for\ every\ }m<k<n \} \\
&= A_{m, n} \cap \bigcap_{m+1\le k \le n-1} B_{m, k}.
\end{align*}
The sets $A_{m,n}$ and $B_{m,k}$ are Borel, hence so is $I_{m,n}\subset\cC_m\times\cC_n$.

Now consider the Borel relation
\[
\widehat I_n
=\{(c_n,(c_m,m))\in\cC_n\times\widetilde\cC_{<n}\colon m<n,\ (c_m,c_n)\in I_{m,n}\},
\quad
\widetilde\cC_{<n}=\bigsqcup_{m<n}(\cC_m\times\{m\}).
\]
Its section over $c_n$ is precisely the finite set of tagged immediate predecessors of $c_n$.
We endow $\widetilde\cC_{<n}$ with the lexicographic Borel order fixed in the definition of (T5).
The ordered finite-section form of the Luzin--Novikov theorem applied to $\widehat I_n$ gives
a Borel counting function
\[
p_n(c_n)=|(\widehat I_n)_{c_n}|
\]
and, on each Borel set $\{p_n=p\}$, Borel maps
\[
L_j(c_n)=(c_{m_j},m_j)\in\widetilde\cC_{<n},\qquad 1\le j\le p,
\]
which list the tagged immediate predecessors in increasing order.

The level maps $m_j$ are Borel as compositions of $L_j$ with the projection
\[
\pi_{\sf lev}\colon\widetilde\cC_{<n}\to\{1,\ldots,n-1\}.
\]
The shape maps are Borel because the map
\[
\widetilde\cC_{<n}\ni(c_m,m)\mapsto \la_m(c_m)\in\cK
\]
is Borel on the finite disjoint union $\widetilde\cC_{<n}$. Finally, the relative-position maps
$c_n\mapsto\rho(c_n,c_{m_j})$ are Borel as compositions with the Borel cocycle
$\rho\colon E_T\to\bR^d$.

Thus, on each Borel set $\{p_n=p\}$, the maps $\bar m_n$, $\bar\lambda_n$, and $\bar\rho_n$
are Borel. Since the pieces $\{p_n=p\}$ are Borel, these piecewise-defined maps combine to Borel maps into the corresponding countable disjoint unions. This proves Lemma~\ref{lemma:p_n}.
\mbox{} \hfill $\Box$

\bigskip
\medskip

\noindent Konstantin Slutsky, Department of Mathematics, Iowa State University, Ames, Iowa, USA \newline
{\tt kslutsky@iastate.edu}

\medskip\noindent Mikhail Sodin,
School of Mathematics, Tel Aviv University, Tel Aviv, Israel,
\newline {\tt sodin@tauex.tau.ac.il}

\medskip\noindent Aron Wennman, Department of Mathematics, KU Leuven, Leuven, Belgium \newline
{\tt aron.wennman@kuleuven.be}
\end{document}